\documentclass[reqno,12pt,a4letter]{amsart}
\usepackage{amsmath, amsxtra, amssymb, latexsym, amscd, amsthm,enumerate}
\usepackage{graphicx, color}
\usepackage[active]{srcltx}
\usepackage[utf8]{inputenc}
\usepackage[mathscr]{euscript}
\usepackage{mathrsfs,cite}
\usepackage[english]{babel}
\usepackage{color,comment}
\usepackage[active]{srcltx}
\usepackage{tikz}
\usepackage{pgfplots}
\usepackage{hyperref}
\newtheorem {theorem}{Theorem}[section]

\newtheorem {proposition}{Proposition}[section]
\newtheorem {lemma}{Lemma}[section]
\newtheorem {example}{Example}[section]
\newtheorem {definition}{Definition}[section]
\newtheorem {remark}{Remark}[section]

\title{On approximate Pareto solutions in nonsmooth interval-valued multiobjective optimization with data uncertainty in constraints}

\author{VU HONG QUAN$^{1}$}
\address{$^1$Faculty of Fundamental Sciences, Thai Nguyen University of Technology, Thai Nguyen, Vietnam}
\email{hongquan@tnut.edu.vn}

\author{DUONG THI VIET AN$^2$}
\address{$^2$Department of Mathematics and Informatics, Thai Nguyen University of Sciences, Thai Nguyen, Vietnam}
\email{andtv@tnus.edu.vn}

\author{NGUYEN VAN TUYEN$^{3,*}$}
\address{$^3$Department of Mathematics, Hanoi Pedagogical University 2, Xuan Hoa, Phuc Yen, Vinh Phuc, Vietnam}
\email{nguyenvantuyen83@hpu2.edu.vn; tuyensp2@yahoo.com}

\thanks{$^*$Corresponding author}

\date{\today}

\keywords{Approximate solutions, Fuzzy KKT conditions, Duality relations,  Interval-valued multiobjective optimization, Robust nonsmooth optimization.}

\subjclass{90C29, 90C46, 90C70,  49J52}

\begin{document}

\maketitle

\begin{abstract} 	This paper deals with  approximate Pareto solutions of a  nonsmooth interval-valued multiobjective optimization problem  with data uncertainty in constraints.  We first introduce some kinds of approximate  Pareto solutions for the robust counterpart \eqref{problem-R} of  the problem in question  by considering the lower-upper interval order relation including: (almost, almost regular) $\mathcal{E}$-Pareto solution and (almost, almost regular) $\mathcal{E}$-quasi Pareto solution. By using a scalar penalty function, we obtain a result on the existence of an almost regular $\mathcal{E}$-Pareto solution of \eqref{problem-R} that satisfies the Karush--Kuhn--Tucker necessary optimality condition up to a given precision. We then establish sufficient conditions and Wolfe-type $\mathcal{E}$-duality relations  for approximate  Pareto solutions of \eqref{problem-R} under the assumption of generalized convexity.  In addition, we present a dual multiobjective problem to the primal one via the $\mathcal{E}$-interval-valued vector Lagrangian function  and examine duality relations. 
\end{abstract}

\section{Introduction}

Due to prediction errors or lack of information, the data of a real-world optimization
problem are often uncertain, i.e., they are not known precisely when the problem is solved, see~\cite{Ben-Nemirovski-98,Ben-Nemirovski-08,Ben-Nemirovski-09}. Robust optimization has emerged as a remarkable
deterministic framework for studying mathematical programming problems with uncertain data. Theoretical and applied aspects in the area of robust optimization have been studied intensively by many researchers; see, e.g.,~\cite{Ben-Nemirovski-09,Chen-19,Chuong-16,Fakhar-19,Gutierrez,Kobis-14,Kobis-15,Kobis-17,Rahimi,Sun-Teo-Zeng-Liu}  and the references therein.

\medskip

Approximate efficient solutions of multiobjective optimization problems can be viewed as feasible points whose objective values display a prescribed error $\mathcal{E}$ in
the optimal values of the vector objective. This concept has been widely discussed
in the works \cite{Chuong-Kim-16,Fakhar-19,Jiao-Liguo-21,Loridan-82,Loridan-84}. Optimality conditions and duality theories of $\mathcal{E}$-efficient solutions and
$\mathcal{E}$-quasi-efficient ones for nonconvex programming problems 
have been studied in \cite{Chuong-Kim-16,Kim-Son-18,Liu-91,Saadati-Oveisiha,Son-Tuyen-Wen-19,Son-09,Son-13,Son-24,TXS-20,Tuyen-2022}.

\medskip

An interval-valued optimization problem is a deterministic optimization model for dealing with uncertain (incomplete) data. Interval-valued optimization problems, where the coefficients of objective and constraint functions are characterized by closed intervals, offer a simplified framework for addressing uncertainty compared to stochastic or fuzzy optimization approaches. That is the main reason why the interval-valued optimization problems have attracted the attention of many researchers, see, e.g., \cite{Chalco-Cano et.al.-13,Hung-Tuan-Tuyen,Hung-Tuyen,Ishibuchi-Tanaka-90,Jennane,Kumar,Moore-1979,Qian,Singh-Dar-15,Singh-Dar-Kim-19,Su-2020,Tung-2019,Tuyen-2021,Wu-07,Wu-09,Wu-09-c} and the references therein.

\medskip
In this paper, we study approximate optimality conditions as well as approximate duality relations of a nonsmooth interval-valued multiobjective optimization problem  with data uncertainty in constraints~\eqref{problem} via its robust counter part~\eqref{problem-R}. To the best of our knowledge, there is no work has been published dealing with approximate optimality conditions
and duality relations in the above-defined optimization problems.  Namely, in this paper we establish necessary/sufficient optimality conditions for almost $\mathcal{E}$ ($\mathcal{E}$-quasi)-Pareto 
solutions of problem~\eqref{problem-R}. These optimality conditions are presented in terms of multipliers and limiting subdifferentials of the related functions. Along with optimality conditions, we propose a dual robust multiobjective optimization problem to problem~\eqref{problem-R}, and
examine $\mathcal{E}$-duality and $\mathcal{E}$-converse like duality relations between them under assumptions of  generalized convexity. We also consider an $\mathcal{E}$-interval-valued vector Lagrangian function and a quasi-$\mathcal{E}$ Pareto saddle point.
In addition, examples are also provided for analyzing and illustrating the obtained results. 

\medskip

Our paper is organized as follows. Section~\ref{Preliminaries} provides some definitions and notation. In Section~\ref{Optimiality-conditions},
we present  optimality conditions for approximate Pareto solutions of  \eqref{problem-R}. Wolfe-type $\mathcal{E}$-duality relations for approximate  Pareto solutions of \eqref{problem-R} under the assumption of generalized convexity are discussed in Section~\ref{Duality-Relations}. In the Section~\ref{section5}, we study a dual multiobjective problem to the primal one via the $\mathcal{E}$-interval-valued vector Lagrangian function  and examine duality relations. Some concluding remarks are
given in Section~\ref{Conclusions}.

\section{PRELIMINARIES} \label{Preliminaries}
We use the following notation and terminology. Fix $n \in {\mathbb{N}}:=\{1, 2, \ldots\}$. The space $\mathbb{R}^n$ is equipped with the usual scalar product and  Euclidean norm. The symbol $B(x,\rho)$ stands for the open ball centered at $x\in \mathbb{R}^n$ with the radius $\rho>0$.
The closed unit ball of $\mathbb{R}^n$ is denoted by $\mathbb{B}$.  We denote the nonnegative orthant in $\mathbb{R}^n$ by  $\mathbb{R}^n_+$.  The topological closure of $S \subset \mathbb{R}^n$ is denoted  by  $\mathrm{cl}\,{S}$.  As usual, the \textit{polar cone} of $S$ is the set
$$S^\circ:= \{x^* \in \mathbb{R}^n : \langle x^*, x \rangle \le 0, \forall x\in S\}.$$ 

\begin{definition}[{see~\cite[Chapter 1]{mor06}}]{\rm  Given $\bar x\in  \mbox{cl}\,S$. The set
		\begin{equation*}
		N(\bar x; S):=\{x^*\in \mathbb{R}^n:\exists
		x^k\stackrel{S}\longrightarrow \bar x, \varepsilon_k\to 0^+, x^*_k\to x^*,
		x^*_k\in {\widehat N_{\varepsilon
				_k}}(x^k; S),\ \ \forall k \in\mathbb{N}\},
		\end{equation*}
		is called the {\em limiting/Mordukhovich normal cone} to $S$ at $\bar x$, where
		\begin{equation*}
		\widehat N_\varepsilon  (x; S):= \bigg\{ {x^*  \in {\mathbb{R}^n} \;:\;\limsup_{u\overset{S} \rightarrow x}
			\frac{{\langle x^* , u - x\rangle }}{{\parallel u - x\parallel }} \leq \varepsilon } \bigg\}
		\end{equation*}
		is the set of  {\em $\varepsilon$-normals} to $S$  at $x$ and  $u\xrightarrow {{S}} x$ means that $u \rightarrow x$ and $u \in S$.
	}
\end{definition}
Let  $\varphi \colon \mathbb{R}^n \to  \overline{\mathbb{R}}$ be an extended-real-valued function, where $\overline{\mathbb{R}}:=[-\infty, +\infty]$. The {\em  epigraph}  and  {\em domain} of $\varphi$ are denoted, respectively, by
\begin{align*}
\mbox{epi }\varphi&:=\{(x, \alpha)\in\mathbb{R}^n\times\mathbb{R} \,:\,  \varphi(x)\leq \alpha\},
\\
\mbox{dom }\varphi &:= \{x\in \mathbb{R}^n \,:\,  \varphi(x) <+\infty \}.
\end{align*}
We say that $\varphi$ is \textit{proper} if
$\varphi(x) > -\infty$ for all $x\in \mathbb{R}^n$ and its domain is nonempty.
\begin{definition}[{see~\cite[Definition 1.77]{mor06}}]{\rm   Consider a function $\varphi: \mathbb{R}^n \to \overline{\mathbb{R}}$ and a point $\bar x \in \mathbb{R}^n$ with $|\varphi (\bar x)| < \infty$.
 The set
		\begin{align*}
		\partial \varphi (\bar x):=\{x^*\in \mathbb{R}^n \,:\, (x^*, -1)\in N((\bar x, \varphi (\bar x)); \mbox{epi }\varphi )\},
		\end{align*}
		is called the {\it limiting/Mordukhovich subdifferential}  of $\varphi$ at $\bar x$. If $|\varphi (\bar x)| = \infty$, then we put $\partial \varphi (\bar x):=\emptyset$.
	}
\end{definition}
The concept of limiting normal cone to sets
and limiting subdifferentials of functions can be linked across. Namely, $ N(\bar x; S)=\partial \delta_S(\bar x),$ where $\delta_S : \mathbb{R}^n \to \mathbb{R} \cup \{+\infty\}$ stands for the indicator function
associated with the set $S$, i.e., the function that takes the value 0 on $\mathbb{R}^n$ and the value $+\infty$ on $\mathbb{R}^n \setminus S.$
\medskip

We now summarize some properties of  the limiting subdifferential that will be used in the next section.

Recall that $\varphi$ is \textit{lower semicontinuous} (l.s.c.) at a point $\bar x$ with $|\varphi(\bar x)| < \infty$ if
$$ \varphi(\bar x) \le \liminf\limits_{x\to \bar x} \varphi(x).$$
One says that $\varphi$ is l.s.c. around $\bar x$ when it is l.s.c. at any point of some neighborhood of $\bar x$.  The \textit{upper semicontinuity} (u.s.c.) of $\varphi$ is defined symmetrically from the lower semicontinuity of $-\varphi$. The function $\varphi$ is called locally Lipschitzian around $\bar x$ there is a
neighborhood $U$ of $\bar x$ and a constant $\ell 
\ge 0$ such that
$ \| \varphi (x)-\varphi(u)\|\le  \ell \|x-u\|$ for all $x,u \in U.$   

\medskip

The nonsmooth version of Fermat’s rule is recalled in the next proposition. 
\begin{proposition}[{\rm see \cite[Proposition 1.114]{mor06}}] \label{Fermat-rule} Let   $\varphi\colon\mathbb{R}^n\to\overline{\mathbb{R}}$  be finite at $\bar x$. If $\bar x$ is a local minimizer of  $\varphi$, then $ 0\in\partial\varphi(\bar x).$
\end{proposition}

The sum rule for the limiting subdifferential is followed~\cite[Theorem 3.36]{mor06}.
\begin{proposition}\label{sum-rule} Let $\varphi_l\colon\mathbb{R}^n\to\overline{\mathbb{R}}$, $l=1, \ldots, p$, $p\geq 2$, be l.s.c. around $\bar x$ and let all but one of these
	functions be locally Lipschitz around $\bar x$. Then we have the following inclusion
	\begin{equation*}
	\partial (\varphi_1+\ldots+\varphi_p) (\bar x)\subset \partial  \varphi_1 (\bar x) +\ldots+\partial \varphi_p (\bar x).
	\end{equation*}
\end{proposition}
The following proposition contains results for computing limiting subdifferentials of maximum functions.
\begin{proposition}[{see \cite[Theorem 3.46]{mor06}}]\label{max-rule}
	Let $\varphi_l\colon\mathbb{R}^n\to\overline{\mathbb{R}}$, $l=1, \ldots, p$,  be  locally Lipschitz around $\bar x$. Then the function
	$ \phi(\cdot):=\max\{\varphi_l(\cdot):l=1, \ldots, p\}$
	is also locally Lipschitz around $\bar x$ and one has
	\begin{equation*}
	\partial \phi(\bar x)\subset \bigcup\bigg\{\partial\bigg(\sum_{l\in I(\bar x)}\lambda_l\varphi_l\bigg)(\bar x)\;:\; (\lambda_1, \ldots, \lambda_p)\in\Lambda(\bar x)\bigg\},
	\end{equation*}
	where $I(\bar x)=\{l\in \{1,...,p\} : \varphi_l (\bar x)=\phi(\bar x) \}$ and 
    $$\Lambda(\bar x):=\big\{(\lambda_1, \ldots, \lambda_p)\;:\; \lambda_l\geq 0, \,\sum_{l=1}^{p}\lambda_l=1, \, \lambda_l[\varphi_l(\bar x)-\phi(\bar x)]=0\big\}.$$
\end{proposition}

Next we recall some definitions and facts in interval analysis.  Let $\mathcal{K}_c$ be the class of all closed and bounded intervals in $\mathbb{R}$, i.e., 
$$\mathcal{K}_c=\{[a^L, a^U]\,:\, a^L, a^U\in\mathbb{R}, a^L\leq a^U\},$$ where $a^L$  and $a^U$  means the lower and upper bounds of $A$, respectively.
Let $A=[a^L, a^U]$ and $B=[b^L, b^U]$ be two intervals in $\mathcal{K}_c$. Then, by definition, we have
\begin{enumerate}[(i)]
	\item $A+B:=\{a+b\,:\, a\in A, b\in B\}=[a^L+b^L, a^U+b^U]$;
	\item $A-B:=\{a-b\,:\, a\in A, b\in B\}=[a^L-b^U, a^U-b^L]$;
	\item For each $k\in\mathbb{R}$,
	\begin{equation*}kA:=\{ka\,:\, a\in A\}=
	\begin{cases}
	[ka^L, ka^U]\ \ \text{if}\ \ k\geq 0,
	\\
	[ka^U, ka^L]\ \ \text{if}\ \ k < 0, 
	\end{cases}
	\end{equation*}
\end{enumerate}
see, e.g., \cite{Alefeld,Moore-1966,Moore-1979}  for more details.
It should be note that if $a^L=a^U$, then $A=[a, a]=a$ is a real number.

The order relation by the lower and upper
bounds of an interval is defined in Definition~\ref{dn2.3}.
\begin{definition}[{see \cite{Kulisch-81,Ishibuchi-Tanaka-90,Wu-07}}]\label{dn2.3}{\rm 
		Let $A=[a^L, a^U]$ and $B=[b^L, b^U]$ be two intervals in $\mathcal{K}_c$. One says that:
		\begin{enumerate}[(i)]
			\item $A\leq_{LU} B$  if and only if $a^L\leq b^L$ and $a^U\leq b^U$.
			
			\item  $A<_{LU} B$ if and only if $A\leq_{LU} B$ and $A\neq B$, or, equivalently, $A<_{LU} B$ if
			\\
			$ 
			\begin{cases}
			a^L<b^L
			\\
			a^U\leq b^U,
			\end{cases}
			$ 
			or \ \ \ \
			$ 
			\begin{cases}
			a^L\leq b^L
			\\
			a^U< b^U,
			\end{cases}
			$ 
			or \ \ \ \
			$ 
			\begin{cases}
			a^L< b^L
			\\
			a^U< b^U.
			\end{cases}
			$ 
			\item $A<^s_{LU} B$ if and only if $a^L<b^L$ and $a^U<b^U$.
		\end{enumerate}
We also define $A\leq_{LU} B$ (resp., $A<_{LU} B$, $A<^s_{LU} B$) if and only if 	$B\geq_{LU} A$ (resp., $B>_{LU} A$, $B>^s_{LU} A$).}
\end{definition}

\section{Approximate optimality conditions}\label{Optimiality-conditions}
For each $j\in J:=\{1, \ldots, p\}$, let $\mathcal{V}_j$ be a nonempty compact subset of $\mathbb{R}^{n_j}$, $n_j\in\mathbb{N}$.  In this paper, we are interested in studying approximate solutions of the following   multiobjective problem with multiple interval-valued objective functions:  
\begin{align}\label{problem}\tag{UMP}
LU-&\mathrm{Min}\;f(x):=\left(f_1(x), \ldots, f_m(x)\right) 
\\
&\text{subject to}\ \ x\in \{x\in S\,:\, g_j(x,v_j)\leq 0, j=1, \ldots p\}, \notag
\end{align}
where $f_i\colon \mathbb{R}^n\to \mathcal{K}_c$, $i\in I:=\{1, \ldots, m\}$, are interval-valued functions and defined respectively by $f_i(x)=[f_i^L(x), f_i^U(x)]$ with $f_i^L,$ $f_i^U\colon\mathbb{R}^n\to \mathbb{R}$  are locally Lipschitz functions satisfying $f_i^L(x)~\leq~f_i^U(x)$  for all $x\in S$ and $i\in I$; $v_j \in \mathcal{V}_j$, $j\in J$, are  uncertain parameters, 
$g_j(\cdot,\cdot)\colon\mathbb{R}^n\times \mathbb{R}^{n_j}\to \mathbb{R}$, $j\in J$, are given functions, and $S$ is a nonempty and closed subset of $\mathbb{R}^n$. 

In what follows, for each $j\in J$, we assume that the function $g_j(\cdot,\cdot)$ satisfies the following conditions:
\begin{itemize}
	\item[($\mathcal{C}_1$)] For a given point $\bar x\in\mathbb{R}^n$, there is $\delta_j^{\bar x}$ such that the function $v_j\in\mathcal{V}_j\mapsto g_j(x, v_j)$ is u.s.c. for each $x\in B(\bar x, \delta_j^{\bar x})$, and the function $g_j(\cdot, v_j), v_j\in\mathcal{V}_j$, are Lipschitz with the same rank $K_j$ on $B(\bar x, \delta_j^{\bar x})$, i.e.,
	\begin{equation*}
		|g_j(x_1, v_j)-g_j(x_2, v_j)|\leq K_j\|x_1-x_2\| \ \ \forall x_1, x_2\in B(\bar x, \delta_j^{\bar x}), \forall v_j\in\mathcal{V}_j. 
	\end{equation*} 
	\item[($\mathcal{C}_2$)]  The multifunction $(x, v_j)\in B(\bar x, \delta_j^{\bar x})\times \mathcal{V}_j\rightrightarrows \partial_x g_j(x, v_j)$ is closed at $(\bar x, \bar v_j)$ for each $\bar v_j\in \mathcal{V}_j(\bar x)$, where the symbol $\partial_x$ stands for the limiting subdifferential operation with respect to variable $x$.
\end{itemize} 

For studying problem \eqref{problem}, one usually associates with it the so-called robust counterpart:
\begin{align}\label{problem-R}\tag{RMP}
	LU-&\mathrm{Min}\;f(x):=\left(f_1(x), \ldots, f_m(x)\right) 
	\\
	&\text{subject to}\ \ x\in \Omega:= \{x\in S\,:\, g_j(x,v_j)\leq 0, \forall v_j\in\mathcal{V}_j,\ \ \forall j\in J\}. \notag
\end{align}

We now introduce approximate solutions of \eqref{problem-R} with respect to $LU$ interval order relation.  Let $\epsilon_i^L$,  $\epsilon_i^U$, $i\in I$, be fixed  real numbers satisfying $0\leq \epsilon_i^L\leq \epsilon_i^U$ for all $i\in I$  and $\sum\limits_{i\in I}(\epsilon_i^L+\epsilon_i^U)>0$. Put $\theta:=\sum\limits_{i\in I}(\epsilon_i^L+\epsilon_i^U)$ and   $\mathcal{E}:=(\mathcal{E}_1, \ldots, \mathcal{E}_m)$, where $\mathcal{E}_i:=[\epsilon_i^L, \epsilon_i^U]$. For each $j\in J$, put $g_j(x):=\max\limits_{v_j\in\mathcal{V}_j} g_j(x, v_j)$. Then, one has
\begin{equation*}
	\Omega=\{x\in S\,:\, g_j(x)\leq 0,\ \ j\in J\}.
\end{equation*}
The $\mathcal{E}$-feasible set $\Omega_{\mathcal{E}}$ is defined by
\begin{equation*}
\Omega_{\mathcal{E}}:=\{x\in S\,:\, g_j(x)\leq \sqrt{\theta},\ \ j\in J\}.
\end{equation*}
Clearly, $\Omega\subset \Omega_{\mathcal{E}}$.

For a given point $\bar x\in \mathbb{R}^n$, the active indices in $\mathcal{V}_j$ at $\bar x$ is denoted by  $\mathcal{V}_j(\bar x)$ and defined as 
\begin{equation*}
\mathcal{V}_j(\bar x):=\{v_j\in \mathcal{V}_j\,:\, g_j(\bar x, v_j)=g_j(\bar x)\}.
\end{equation*}

The following concepts of solutions can be found in the literature; see, e.g.,~\cite{Loridan-82,Loridan-84}.
\begin{definition}\label{Defi-solution-n1}{\rm
		Let $\bar x\in \Omega$. We say that:
		\begin{enumerate}[(i)]
			\item $\bar x$ is an {\em $\mathcal{E}$-Pareto solution} of \eqref{problem-R}, denoted by $\bar x\in \mathcal{E}$-$\mathcal{S}\eqref{problem-R}$, if there is no $x\in \Omega$ such that
			\begin{equation}\label{equa-1-n}
				\begin{cases}
					f_i(x)\leq_{LU} f_i(\bar x) -\mathcal{E}_i,\ \ &\forall i\in I,
					\\
					f_k(x)<_{LU} f_k(\bar x) -\mathcal{E}_k,\ \ &\text{for at least one}\ \ k\in I.
				\end{cases} 
			\end{equation}
			\item $\bar x$ is an {\em  $\mathcal{E}$-quasi Pareto solution} of \eqref{problem-R}, denoted by $\bar x\in \mathcal{E}$-$\mathcal{S}^q\eqref{problem-R}$, if there is no $x\in \Omega$ such that
			\begin{equation}\label{equa-5-n}
			\begin{cases}
			f_i(x)\leq_{LU} f_i(\bar x) -\frac{\mathcal{E}_i}{\sqrt{\theta}}\|x-\bar x\|,\ \ &\forall i\in I,
			\\
			f_k(x)<_{LU} f_k(\bar x) -\frac{\mathcal{E}_k}{\sqrt{\theta}}\|x-\bar x\|,\ \ &\text{for at least one}\ \ k\in I.
			\end{cases} 
			\end{equation}
		\end{enumerate}
	}
\end{definition}

We now introduce the concepts of  almost approximate Pareto solutions to \eqref{problem-R}. These concepts are inspired from \cite{Loridan-82}.
\begin{definition}\label{Definition-2}
\rm 
Let $\bar x\in \Omega_{\mathcal{E}}$.  We say that:
\begin{enumerate} [\rm(i)]
	\item  $\bar x$ is an {\em almost  $\mathcal{E}$-Pareto solution} (respectively, {\em almost  $\mathcal{E}$-quasi Pareto solution}) of \eqref{problem-R}, denoted by $\bar x\in \mathcal{E}$-$\mathcal{S}_{a}\eqref{problem-R}$ (respectively, $\bar x\in \mathcal{E}$-$\mathcal{S}^q_{a}\eqref{problem-R}$) if there is no $x\in \Omega$  satisfying \eqref{equa-1-n} (respectively,  \eqref{equa-5-n}).  
	\item $\bar x$ is an {\em almost regular $\mathcal{E}$-Pareto solution}  of \eqref{problem-R} if it is both an almost $\mathcal{E}$-Pareto solution and an almost  $\mathcal{E}$-quasi Pareto solution of \eqref{problem-R}.   
\end{enumerate}
\end{definition}
The concepts of an $\mathcal{E}$-Pareto solution and an almost $\mathcal{E}$-Pareto one are independent. The following example shows that  an almost $\mathcal{E}$-Pareto solution may not be an  $\mathcal{E}$-Pareto one.
\begin{example}\rm 
Let $f := (f_1,f_2),$ where  $f_1,f_2: \mathbb{R}^2 \to \mathcal{K}_C$ are defined by
$$ f_1(x)=[x_1-1, x_2+1],\ f_2(x)=[-x_1, -x_2+2], \ \forall x=(x_1,x_2)\in \mathbb{R}^2,$$
and let $g_j: \mathbb{R}^2 \times \mathbb{R} \to \mathbb{R}$, $j=1,2$, be given by
 $$ g_1(x,v_1)=x_1-v_1,\ g_2(x, v_2)=x_2- v_2,\ \mbox{with}\   v_1, v_2 \in [1, 2].$$
 Consider the problem~\eqref{problem-R} with $$
S=\{(x_1,x_2) \in \mathbb{R}^2\ : x_1 \geq 0, x_2 \geq 0 \}.$$
 Then, one has
$g_1(x)=x_1-1,\  g_2(x)=x_2-1$, and
$ \Omega:= \{x\in S\,:x_1\leq 1,x_2\leq 1\}$.
For $\mathcal{E}_1=\mathcal{E}_2=\left[0;  \tfrac{1}{2} \right].$  Then, $\theta=1,$ and
$\Omega_\mathcal{E}:= \{x\in S\,:x_1\leq 2,x_2\leq 2\}.$

Take $z=(1,2)$, we see that $z \in \Omega_\mathcal{E}$. We now show that $z$ is an almost $\mathcal{E}$-Pareto solution   of~\eqref{problem-R}. Indeed, suppose on the contrary that  $z$  is not an  almost  $\mathcal{E}$-Pareto solution of~\eqref{problem-R}. Then, by definition, there exists an $x=(x_1,x_2)\in\Omega$ such that 
\begin{align} \label{ct3}
	\begin{cases}
		f_i(x)\leq_{LU} f_i(z) -\mathcal{E}_i,\ \ &\forall i\in I:=\{1,2\},
		\\ 
		f_k(x)<_{LU} f_k(z) -\mathcal{E}_k,\ \ &\text{for at least}\ \ k\in I. 
	\end{cases} 
\end{align} 
The first inequality of~\eqref{ct3} means that 
\begin{equation*} 
\begin{cases}
		f^L_i(x)\leq  f^L_i(z) -\mathcal{E}^U_i \\
		f^U_i(x)\leq  f^U_i(z) -\mathcal{E}^L_i
\end{cases} 
\Leftrightarrow 
        \begin{cases}
		x_1-1\leq \displaystyle -\frac{1}{2} \\
		x_2+1\leq 3- 0 \\
        -x_1\leq -1- \displaystyle\frac{1}{2} \\
		-x_2+2\leq  0, 
	\end{cases} 
\end{equation*}
or, equivalently,
\begin{equation*}
    \begin{cases}
		x_1\leq \dfrac{1}{2} \\
        x_1\geq \dfrac{3}{2} \\
		x_2\geq 2\\
        x_2 \le 2,
	\end{cases} 
\end{equation*} 
which is impossible. Thus, $z$ is an almost $\mathcal{E}$-Pareto solution
of~\eqref{problem-R}. However,  it is clear that $z\notin \Omega$ and thus $z$  is not an $\mathcal{E}$-Pareto solution of~\eqref{problem-R}.
\end{example}

Let us consider the following scalar optimization problem
\begin{equation}\label{problem-S}\tag{P}
\min \left\{\varphi(x):=\sum_{i\in I}[f_i^L(x)+f_i^U(x)]\,:\, x\in \Omega\right\}.
\end{equation}
\begin{definition}
\rm Let $z\in \Omega_{\mathcal{E}}$. We say that:
\begin{enumerate}[\rm(i)]
\item $z$ is an {\em almost $\theta$-solution} of \eqref{problem-S} if $\varphi(z)\leq \varphi(x)+\theta$ for all $x\in\Omega$.
\item  $z$ is an {\em almost $\theta$-quasi solution} of \eqref{problem-S} if $\varphi(z)\leq \varphi(x)+\sqrt{\theta}\|x-z\|$ for all $x\in\Omega$.
\item $z$ is an {\em almost regular $\theta$-solution} of \eqref{problem-S} if it is both an almost $\theta$-solution and  an almost $\theta$-quasi solution of \eqref{problem-S}.
\end{enumerate}
\end{definition}
The following lemma gives the relationship between  solutions of~\eqref{problem-S} and~\eqref{problem-R}.
\begin{lemma}\label{lemma-3.2}
Let $z\in \Omega_{\mathcal{E}}$. Then the following assertions hold:
\begin{enumerate}[\rm(i)]
	\item If $z$ is an almost $\theta$-solution of \eqref{problem-S}, then it is an  almost   $\mathcal{E}$-Pareto solution of \eqref{problem-R}.
	
	\item  If $z$ is an almost $\theta$-quasi solution of \eqref{problem-S}, then it is an  almost  $\mathcal{E}$-quasi Pareto solution of \eqref{problem-R}.
	
	\item If $z$ is an almost regular $\theta$-solution of \eqref{problem-S}, then it is an  almost regular $\mathcal{E}$-Pareto solution of \eqref{problem-R}.
\end{enumerate}
\end{lemma}
\begin{proof}
(i) Let $z$ be an almost $\theta$-solution of \eqref{problem-S}. Suppose on the contrary that  $z$  is not an  almost  $\mathcal{E}$-Pareto solution of \eqref{problem-R}. Then, by definition, one can find $x\in\Omega$ and $k\in I$ such that 
\begin{equation*} 
	\begin{cases}
		f_i(x)\leq_{LU} f_i(z) -\mathcal{E}_i,\,\forall i\in I,
		\\
		f_k(x)<_{LU} f_k(z) -\mathcal{E}_k.
	\end{cases} 
\end{equation*} 
This means that 
\begin{equation*} 
	\begin{cases}
		f^L_i(x)\leq  f^L_i(z) -\epsilon^U_i, \\
		f^U_i(x)\leq  f^U_i(z) -\epsilon^L_i, 
	\end{cases} 
\end{equation*} 
for all $i\in I$ and $f^L_k(x)<  f^L_k(z) -\epsilon^U_k$, or $f^U_k(x)<  f^U_k(z) -\epsilon^L_k$. Hence,
\begin{equation*}
	\sum_{i\in I}[f_i^L(x)+f_i^U(x)]<\sum_{i\in I}[f_i^L(z)+f_i^U(z)]-\sum_{i\in I}(\epsilon_i^L+\epsilon_i^U),
\end{equation*}
or, equivalently,
\begin{equation*}
\varphi(x)<\varphi (z)-\theta,
\end{equation*}
contrary to the assumption that $z$ is an almost $\theta$-solution of \eqref{problem-S}.

The proofs of the remaining assertions follow similarly to the proof of assertion (i) and are therefore omitted.
\end{proof}
Under the conditions ($\mathcal{C}_1$) and ($\mathcal{C}_2$), we obtain the following upper estimate for the subdifferential of the functions $g_j$. This result is given in the proof of \cite[Theorem 3.3]{Chuong-16}.
\begin{lemma}[{see \cite[Theorem 3.3]{Chuong-16}}]\label{sup-rule}
	Under the conditions ($\mathcal{C}_1$) and ($\mathcal{C}_2$) we have
	\begin{equation*}
		\partial g_j(\bar x)\subset\mathrm{co}\, \{\cup\partial_xg_j(\bar x, v_j)\,:\, v_j\in\mathcal{V}_j(\bar x)\} \ \ \forall j\in J,
	\end{equation*}
	where $\mathrm{co}\, A$ means the convex hull of the set $A$.
\end{lemma}

To obtain necessary optimality conditions for approximate Pareto solutions of ~\eqref{problem-R}, we first characterize  necessary optimality conditions for  approximate solutions of~\eqref{problem-S}.
\begin{theorem}\label{Necessary-P} If the objective function $\varphi$ of problem \eqref{problem-S} is bounded from below on $S$, then there exist an almost regular $\theta$-solution $z$ and $\lambda\in\mathbb{R}^p_+$ such that
\begin{align}
&0\in\partial \varphi (z)+\sum_{j\in J}\lambda_j\mathrm{co}\,\left\{\partial_xg_j(z, v_j)\,:\, v_j\in \mathcal{V}_j(z)\right\} +N(z; S)+\sqrt{\theta}\mathbb{B},\label{equa-Eke-4}\\
&\lambda_j>0 \ \ \text{for all}\ \ j\in J\ \ \text{satisfying} \ \ 0<g_j(z)\leq \sqrt{\theta}, \label{equa-Eke-5}
\\
&\lambda_j=0 \ \ \text{if}\ \ g_j(z)\leq 0. \label{equa-Eke-6}
\end{align}	
\end{theorem}
\begin{proof}
Consider the following penalized functional
\begin{equation*}
\phi(x):=\varphi(x)+\sum_{j\in J}\frac{1}{r_j}[g_j^+(x)]^2+\delta_S(x), \ \ \forall x\in\mathbb{R}^n,
\end{equation*} 
where $g_j^+(x):=\max\{0, g_j(x)\}$ and $r_j$,  $j\in J$, are real positive numbers.   Clearly, the function $\phi$ is l.s.c. and bounded from below on $\mathbb{R}^n$. Hence, by  the Ekeland variational principle, for $\theta>0$ there exists $z\in \mathbb{R}^n$ such that
\begin{align}
&\phi(z)\leq \phi(x)+\theta, \ \ \forall x\in\mathbb{R}^n,\label{equa-Eke-1}
\\
&\phi(z)\leq \phi(x)+\sqrt{\theta}\|x-z\|, \ \ \forall x\in\mathbb{R}^n. \label{equa-Eke-2}
\end{align} 
It follows from \eqref{equa-Eke-1} that $z\in S$ and
\begin{equation}\label{equa-Eke-3}
	\varphi(z)\leq\phi(z)\leq \varphi(x)+\theta, \ \ \forall x\in \Omega, 
\end{equation}
or, equivalently,
\begin{equation*}
	\varphi(z)\leq\varphi(z)+\sum_{j\in J}\frac{1}{r_j}[g_j^+(z)]^2\leq \inf_{x\in\Omega}\varphi(x)+\theta. 
\end{equation*}
This and the fact that $\inf_{x\in S}\varphi(x)\leq \varphi(z)$ imply that
\begin{equation*}
\sum_{j\in J}\frac{1}{r_j}[g_j^+(z)]^2\leq \inf_{x\in\Omega}\varphi(x)+\theta -\inf_{x\in S}\varphi(x).
\end{equation*}
Hence, for each $j\in J$, $g_j^+(z)$ tends to $0$ as $r_j\to 0$. Consequently, there exists $r_j^\theta$  such that $g^+_j(z)<\sqrt{\theta}$ for all $r_j\leq r_j^\theta$.  Hence, with such a choice of $r_j$, $j\in J,$  we can define a point $z\in \Omega_{\mathcal{E}}$. Furthermore, from \eqref{equa-Eke-2},  we have
\begin{equation*}
\varphi(z)\leq \varphi(x)+\sqrt{\theta}\|x-z\| \ \ \forall x\in \Omega.
\end{equation*} 
This together with  \eqref{equa-Eke-3} imply  that $z$ is an almost regular $\theta$-solution of \eqref{problem-S}. Now, it follows from \eqref{equa-Eke-2} that   $z$ is a minimizer of the scalar optimization  problem
\begin{equation*}
\min\{\phi(x)+\sqrt{\theta}\|x-z\|\,:\, x\in\mathbb{R}^n\}.
\end{equation*}
By Propositions \ref{Fermat-rule}, \ref{sum-rule}, and the fact that $\partial\|\cdot-z\|(z)=\mathbb{B}$ (see \cite[Example 4, p. 198]{Ioffe79})), we get 
\begin{equation*}
	0\in \partial\phi(z) +\sqrt{\theta}\mathbb{B}. 
\end{equation*}
From the Lipschitz continuity of functions $\varphi$ and $g_j^+$, $j\in J$, one has
\begin{align*}
\partial \phi(z)&\subset \partial\varphi(z)+\sum_{j\in J}\frac{1}{r_j}\partial[g_j^+(\cdot)]^2(z)+N(z; S)
\\
&\subset \partial\varphi(z)+2\sum_{j\in J}\frac{1}{r_j}g_j^+(z)\partial g_j(z)+N(z; S)
\\
&\subset \partial\varphi(z)+2\sum_{j\in J}\frac{1}{r_j}g_j^+(z)\mathrm{co}\,\{\partial_xg_j(z, v_j)\,:\, v_j\in \mathcal{V}_j(z)\} +N(z; S),
\end{align*}
where the last inclusion is deduced from Lemma \ref{sup-rule}. Hence, by setting $\lambda_j:=\dfrac{2}{r_j}g_j^+(z)$, $j\in J$, we obtain conditions \eqref{equa-Eke-4}--\eqref{equa-Eke-6}. The proof is complete.
\end{proof}
Building upon Theorem \ref{Necessary-P}, Lemma \ref{lemma-3.2}, and Proposition \ref{sum-rule}, we are now in a position to derive the necessary optimality conditions for~\eqref{problem-R}.
\begin{theorem}\label{Existence-Theorem} If the objective function $\varphi$ of problem \eqref{problem-S} is bounded from below on $S$, then there exist an almost regular   $\mathcal{E}$-solution $z$ of \eqref{problem-R} and $\lambda\in\mathbb{R}^p_+$ such that
\begin{align}
	&0\in\sum_{i\in I}[\partial f_i^L(z)+\partial f_i^U(z)]+\sum_{j\in J}\lambda_j\mathrm{co}\,\{\partial_xg_j(z, v_j)\,:\, v_j\in \mathcal{V}_j(z)\} +N(z; S)+\sqrt{\theta}\mathbb{B}, \label{equation-8-n} \\
	&\lambda_j>0 \ \ \text{for all}\ \ j\in J\ \ \text{satisfying} \ \ 0<g_j(z)\leq \sqrt{\theta}, \label{equation-9-n} 
	\\
	&\lambda_j=0 \ \ \text{if}\ \ g_j(z)\leq 0.  \label{equation-5-n}
\end{align}	
\end{theorem}

\begin{definition}
\rm 
A pair $(z, \lambda)\in \Omega_{\mathcal{E}}\times\mathbb{R}^p_+$ is called a {\em Karush--Kuhn--Tucker (KKT)  pair up to $\mathcal{E}$} of \eqref{problem-R} if it satisfies conditions \eqref{equation-8-n}--\eqref{equation-5-n}. 
\end{definition}

Next, we present sufficient conditions for approximate quasi Pareto solutions of \eqref{problem-R}. To do this, we need to introduce concepts of  generalized convexity   for a family of locally Lipschitz functions. The following definition is motivated from  \cite{Fakhar-19,Hung-Tuan-Tuyen}.

Wet set $g:=(g_1, \ldots, g_p)$ for convenience in the sequel.
\begin{definition}
{\rm  
 We say that $(\varphi, g)$ is {\em $\theta$-pseudo-quasi generalized  convex on $S$} at $z\in S$ if for any $x\in S$, $z_i^{*L}\in\partial f_i^L(z)$,  $z_i^{*U}\in\partial f_i^U(z)$, $i\in I$, and $x^*_{v_j}\in\partial_x g_j(z, v_j)$, $v_j\in \mathcal{V}_j(z)$, $j\in J$, there exists $\omega\in [N(z; S)]^\circ$ satisfying
\begin{equation}\label{equ:3}
\begin{split}
&\sum_{i\in I}[\langle z_i^{*L}, \omega\rangle+\langle z_i^{*U}, \omega\rangle] + \sqrt{\theta}\|x-z\|\geq 0 \Rightarrow \varphi(x) \geq \varphi(z)-\sqrt{\theta}\|x-z\|, 
\\
&g_j(x, v_j)\leq g_j(z, v_j) \Rightarrow \langle x^*_{v_j}, \omega\rangle\leq 0,\ \  v_j\in \mathcal{V}_j(z), j\in J,\ \ \text{and}\ \ \\
& \langle b^*,\omega\rangle\leq \|x-z\|,\ \ \forall b^*\in \mathbb{B}.
\end{split}
\end{equation}
}
\end{definition}

\begin{theorem}[{Sufficient condition of type I}]\label{Sufficient-Theorem-I} Let $(z, \lambda)\in\Omega_{\mathcal{E}}\times\mathbb{R}^p_+$ be a KKT pair up to $\mathcal{E}$ of \eqref{problem-R}. If $(\varphi, g)$ is $\theta$-pseudo-quasi generalized convex on $S$ at $z$, then $z\in \mathcal{E}$-$\mathcal{S}_{a}^{q}\eqref{problem-R}$.
\end{theorem}	
\begin{proof} Since $(z, \lambda)$ is a KKT pair up to $\mathcal{E}$ of \eqref{problem-R}, there exist $z_i^{*L}\in\partial f_i^L(z)$,  $z_i^{*U}\in\partial f_i^U(z)$, $i\in I$,  $\lambda_{jk}\geq 0$, $x^*_{jk}\in\partial_x g_j(z, v_{jk})$, $v_{jk}\in\mathcal{V}_j(z)$, $k=1, \ldots, k_j$, $k_j\in\mathbb{N}$, $\sum_{k=1}^{k_j}\lambda_{jk}=1$,    $b^*\in \mathbb{B}$, and $\omega^* \in N(z; S)$ such that
\begin{equation*}
\sum_{i\in I}( z_i^{*L}+  z_i^{*U})+\sum_{j\in J}\lambda_j\bigg(\sum_{k=1}^{k_j}\lambda_{jk}x^*_{jk}\bigg)+\sqrt{\theta}b^*+\omega^*=0.
\end{equation*}	
Hence,
\begin{equation}\label{equ:5}
\sum_{i\in I}( z_i^{*L}+  z_i^{*U})+\sum_{j\in J} \sum_{k=1}^{k_j}\lambda_j\lambda_{jk}x^*_{jk} +\sqrt{\theta}b^*=-\omega^*.
\end{equation}
Suppose on the contrary that $z$ is not an almost  $\mathcal{E}$-quasi Pareto solution of \eqref{problem-R}, then, by Lemma \ref{lemma-3.2}, it is not an almost $\theta$-quasi solution of \eqref{problem-S}. Hence,    there exists $x\in \Omega$ such that  
\begin{equation}\label{equation-4-n}
\varphi (x)<\varphi(z)-\sqrt{\theta}\|x-z\|.
\end{equation}
By the $\theta$-pseudo-quasi generalized  convexity of $(\varphi, g)$ at $z\in S$, there is $\omega\in [N(z; S)]^\circ$ satisfying \eqref{equ:3}. Thus, we deduce from \eqref{equ:3} and \eqref{equation-4-n} that
\begin{equation}\label{equation-6-n}
\sum_{i\in I}[\langle z_i^{*L}, \omega\rangle+\langle z_i^{*U}, \omega\rangle] + \sqrt{\theta}\|x-z\|< 0.
\end{equation}
For each $j\in J$, by \eqref{equation-5-n}, we see that $\lambda_j=0$ whenever $g_j(z)\leq 0$ and $\lambda_j>0$ if otherwise. Furthermore, when $g_j(z)>0$, then due to the fact that $x\in\Omega$, one has
 \begin{equation*}
 g_j(x, v_{jk})\leq 0<g_j(z)=g_j(z, v_{jk}), \ \ \forall v_{jk}\in\mathcal{V}_j(z)
 \end{equation*}
 and so we get
\begin{equation*}
\langle x_{jk}^*, \omega\rangle\leq 0.
\end{equation*}
This implies that
\begin{equation}\label{equation-7-n}
\sum_{j\in J} \sum_{k=1}^{k_j}\lambda_j\lambda_{jk}\langle x^*_{jk}, \omega \rangle\leq 0.
\end{equation}
Since  $\omega\in [N(z; S)]^\circ$, \eqref{equ:3}, \eqref{equ:5}, \eqref{equation-6-n}, and~\eqref{equation-7-n} we obtain
\begin{align*}
0&\leq \langle -\omega^*,\omega\rangle= \sum_{i\in I}[\langle z_i^{*L}, \omega\rangle+\langle z_i^{*U}, \omega\rangle] +\sum_{j\in J} \sum_{k=1}^{k_j}\lambda_j\lambda_{jk}\langle x^*_{jk}, \omega \rangle+ \sqrt{\theta}\langle b^*,\omega\rangle 
\\
&<-\sqrt{\theta}\|x-z\|+\sqrt{\theta}\langle b^*,\omega\rangle\leq 0,
\end{align*} 
a contradiction. Hence, $z\in \mathcal{E}$-$\mathcal{S}_{a}^{q}\eqref{problem-R}$.   The proof  is complete.  
\end{proof} 

\begin{example} \rm Let	$f_1(x)=[2x_1-1, 3x_2+2]$, $f_2(x)=[x_1+1,x_2+4]$ and
		$$S=\{(x_1,x_2)\in\mathbb{R}^2:x_1\ge 0, x_2\ge 0\}, \ \forall x=(x_1,x_2)\in \mathbb{R}^2.$$
	Consider the problem~\eqref{problem-R}	with  $ g_1(x,v_1)=-x_1+v_1, \  g_2(x,v_2)=-2x_2+v_2,$ for all $v_1,v_2\in[0,1].$
	
        We first see that $f_1, f_2, g_1, g_2$ are Lipschitz, 
		$ g_1(x)=-x_1+1\le 0,$ $g_2(x)=-2x_2+1\le 0, $	 
        and $\Omega=\left\{x\in\mathbb{R}^2: x_1\ge 1, x_2\ge\frac{1}{2}\right\}.$  Put $\mathcal{E}_1=[0,\frac{1}{2}], \mathcal{E}_2=[0,\frac{1}{2}]$. Then, one has $\theta=1$ and 
		$ \Omega_{\mathcal{E}}=\{x\in\mathbb{R}:x_1\ge 0, x_2 \ge 0\}.$ 
        Let $z=\left(0,0\right)$. It is not hard to see that $(z,\lambda) \in \Omega_{\mathcal{E}}\times \mathbb{R}^2_+$, where $\lambda=(3,1)$ is a KKT pair up to $\mathcal{E}$ of~\eqref{problem-R}. We next verify the $\theta$-pseudo quasi generalized convexity on $S$ at $z$.  By simple calculations, we get
		\begin{align*}
			&\partial f_1^L(z)=\{2\},\ \partial f_1^U(z)=\{3\}, \ \partial f_2^L(z)=\{1\},\ \partial f_2^U(z)=\{1\},\\
			& \partial_xg_1(z,v_1)=\{-1\}, \ \partial_xg_2(z,v_2)=\{-2\}, \ N(z;S)=\{(0,0)\}.
		\end{align*}
So, we can find $w=(0,0)\in [N(z;S)]^\circ $ satisfying~\eqref{equ:3}. The latter means that all the assumptions of Theorem~\ref{Sufficient-Theorem-I} are fulfilled.

		Consider  the scalar optimization problem~\eqref{problem-S} with
        $\varphi(x)=3x_1+4x_2+6$ and the constraint set $\Omega=\{x\in\mathbb{R}^2: x_1\ge 1, x_2\ge \frac{1}{2}\}$. Since	$\varphi(z)\le \varphi(x)+\sqrt{\theta}\|x-z\|\ $ for all $x\in\Omega$, it follows that $z=(0,0)$ is an almost $\theta$-quasi solution of~\eqref{problem-S}.
		Thus, by applying Lemma~\ref{lemma-3.2} we conclude that $z$ is an almost $\mathcal{E}$-quasi Pareto solution of~\eqref{problem-R}.
\end{example}
The following definitions are inspired from \cite{Chuong-16,Fakhar-19,Hung-Tuan-Tuyen}.
\begin{definition}
	{\rm  We say that $(f, g)$ is {\em generalized  convex on $S$} at $z\in S$ if for any $x\in S$, $z_i^{*L}\in\partial f_i^L(z)$,  $z_i^{*U}\in\partial f_i^U(z)$, $i\in I$, and $x^*_{v_j}\in\partial_x g_j(z, v_j)$, $v_j\in \mathcal{V}_j(z)$, $j\in J$, there exists $\omega\in [N(z; S)]^\circ$ satisfying
			\begin{equation*} 
				\begin{split}
					&f_i^L(x) -  f_i^L(z)\geq \langle z_i^{*L}, \omega\rangle, \ \ \forall i\in I,
					\\
					&f_i^U(x) -  f_i^U(z)\geq \langle z_i^{*U}, \omega\rangle, \ \ \forall i\in I,
					\\
					&g_j(x, v_j)- g_j(z, v_j)\geq  \langle x^*_{v_j}, \omega\rangle,\ \  v_j\in \mathcal{V}_j(z), j\in J,\ \ \text{and}\ \ \\
					& \langle b^*,\omega\rangle\leq \|x-z\|,\ \ \forall b^*\in\mathbb{B}.
				\end{split}
			\end{equation*}
	}
\end{definition}

\begin{definition}
	{\rm  We say that $(f, g)$ is {\em $\mathcal{E}$-pseudo-quasi generalized  convex on $S$} at $z\in S$ if for any $x\in S$, $z_i^{*L}\in\partial f_i^L(z)$,  $z_i^{*U}\in\partial f_i^U(z)$, $i\in I$, and $x^*_{v_j}\in\partial_x g_j(z, v_j)$, $v_j\in \mathcal{V}_j(z)$, $j\in J$, there exists $\omega\in [N(z; S)]^\circ$ satisfying
			\begin{equation}\label{equ:3-n}
				\begin{split}
					&\langle z_i^{*L}, \omega\rangle + \frac{\epsilon_i^U}{\sqrt{\theta}}\|x-z\|\geq 0 \Rightarrow f_i^L(x) \geq  f_i^L(z)-\frac{\epsilon_i^U}{\sqrt{\theta}}\|x-z\|, \ \ \forall i\in I,
			\\
			&\langle z_i^{*U}, \omega\rangle + \frac{\epsilon_i^L}{\sqrt{\theta}}\|x-z\|\geq 0 \Rightarrow f_i^U(x) \geq  f_i^U(z)-\frac{\epsilon_i^L}{\sqrt{\theta}}\|x-z\|, \ \ \forall i\in I,
					\\
					&g_j(x, v_j)\leq g_j(z, v_j) \Rightarrow \langle x^*_{v_j}, \omega\rangle\leq 0,\ \  v_j\in \mathcal{V}_j(z), j\in J,\ \ \text{and}\ \ \\
					& \langle b^*,\omega\rangle\leq \|x-z\|,\ \ \forall b^*\in\mathbb{B}.
				\end{split}
			\end{equation}
	}
\end{definition}
\begin{remark}\rm 
   It is straightforward to observe that if $(f,g)$ is  generalized convex on $S$ at $\bar x$, then, for any $\mathcal{E},$ $(f,g)$  is $\mathcal{E}$-pseudo-quasi generalized convex on $S$ at $\bar x$. Moreover, the class of $\mathcal{E}$-pseudo-quasi generalized convex functions is properly lager than the one of generalized convex functions.
To demonstrate this point, let us consider the following example.
\end{remark}
\begin{example}\label{ex3.3} \rm 
Consider the problem \eqref{problem-R} with
		$f_1(x)=[x,x+2]$,
		$f_2(x)=[x,x+1],$ $S=\{x\in\mathbb{R}:x\ge 0\}$, and
		$g(x,v)=-x^2-v+1, v\in[0,1]$. Take $\mathcal{E}_1=[0,\frac{1}{2}], \mathcal{E}_2=[0,\frac{1}{2}]$.  We will prove $(f, g)$ is $\mathcal{E}$-pseudo-quasi generalized convex at $z=0$ but not generalized convex at this point. Indeed, we first see that
        \begin{align*}
           & \partial f_1^L(z)=\{1\}, \ \partial f_1^U(z)=\{1\},\\
          &  \partial f_2^L(z)=\{1\}, \ \partial f_2^U(z)=\{1\},\\
          &  \partial_xg(z,v)=\{0\}, \ N(z;S)=(-\infty, 0].
        \end{align*}
    Then, for any $x\in S$, $z_1^{*L}=1,$   $z_1^{*U}=1,$   $z_2^{*L}=1,$   $z_2^{*U}=1,$ and $x^*_v=0$, we always find $w\in [N(z;S)]^\circ$   such that~\eqref{equ:3-n} holds. In the other words, $(f, g)$ is $\mathcal{E}$-pseudo-quasi generalized convex at $z=0$. However, $(f, g)$ is not generalized convex on $S$ at $z =0$ since for any $x\in S$,
		$$g(x, v)- g(z, v)\geq  \langle x^*_{jk}, \omega\rangle \Leftrightarrow -x^2\geq  0,$$
       a contradiction. 
\end{example}

\begin{theorem}[{Sufficient condition of type II}]\label{Sufficient-Theorem-II}  Let $(z, \lambda)\in\Omega_{\mathcal{E}}\times\mathbb{R}^p_+$ be a KKT pair up to $\mathcal{E}$ of \eqref{problem-R}. If $(f, g)$ is $\mathcal{E}$-pseudo-quasi generalized convex on $S$ at $z$, then $z\in \mathcal{E}$-$\mathcal{S}_{a}^{q}\eqref{problem-R}$.
\end{theorem}	
\begin{proof} Since $(z, \lambda)$ is a KKT pair up to $\mathcal{E}$ of \eqref{problem-R},  there exist $z_i^{*L}\in\partial f_i^L(z)$,  $z_i^{*U}\in\partial f_i^U(z)$, $i\in I$,  $\lambda_{jk}\geq 0$, $x^*_{jk}\in\partial_x g_j(z, v_{jk})$, $v_{jk}\in\mathcal{V}_j(z)$, $k=1, \ldots, k_j$, $k_j\in\mathbb{N}$, $\sum_{k=1}^{k_j}\lambda_{jk}=1$,    $b^*\in  \mathbb{B}$, and $\omega^* \in N(z; S)$ such that 
\begin{equation*} 
	\sum_{i\in I}( z_i^{*L}+  z_i^{*U})+\sum_{j\in J} \sum_{k=1}^{k_j}\lambda_j\lambda_{jk}x^*_{jk} +\sqrt{\theta}b^*=-\omega^*.
\end{equation*}
Suppose on the contrary that $z\notin \mathcal{E}$-$\mathcal{S}_{a}^{q}\eqref{problem-R}$, then there exists a point $x\in \Omega$ such that
\begin{equation*}
f_i(x)\leq_{LU} f_i(z) -\frac{\mathcal{E}_i}{\sqrt{\theta}}\|x-z\|,\ \ \forall i\in I,
\end{equation*}
where at least one of the inequalities is strict, or, equivalently,
\begin{equation}\label{equ:8-n}
f^L_i(x)\leq f^L_i(z)-\frac{\epsilon^U_i}{\sqrt{\theta}}\|x-z\| \ \ \text{and}\ \ f^U_i(x)\leq f^U_i(z)-\frac{\epsilon^L_i}{\sqrt{\theta}}\|x-z\|, \ \ \forall i\in I,
\end{equation}
with at least one strict  inequality. By the $\mathcal{E}$-pseudo-quasi generalized  convexity of $(f, g)$  at $z\in S$, there exists $\omega\in [N(z; S)]^\circ$ satisfying \eqref{equ:3-n}. By \eqref{equ:3-n} and \eqref{equ:8-n}, one has
\begin{equation*}
	\langle z_i^{*L}, \omega\rangle + \frac{\epsilon_i^U}{\sqrt{\theta}}\|x-z\|\leq  0 \ \ \text{and}\ \ \langle z_i^{*U}, \omega\rangle + \frac{\epsilon_i^L}{\sqrt{\theta}}\|x-z\|\leq  0, \ \ \forall i\in I,
\end{equation*}
with at least one strict inequality. Hence,
\begin{equation*}
\sum_{i\in I}[\langle z_i^{*L}, \omega\rangle+\langle z_i^{*U}, \omega\rangle] + \sqrt{\theta}\|x-z\|< 0.	
\end{equation*}
By the same argument after \eqref{equation-6-n}  in the proof of Theorem \ref{Sufficient-Theorem-I}, we arrive at a contradiction. 
\end{proof}
We close this section with an illustrated example.
 \begin{example}\rm Let $f=(f_1,f_2)$ where
		$f_1(x)=[x,x+2]$,
		$f_2(x)=[x,x+1], \forall x\in \mathbb{R}$. Consider the problem~\eqref{problem-R} with $S=\{x\in\mathbb{R}:x\ge 0\}$, and $g(x,v)=-x+v-1, v\in[0,2]$. Then, we have
		$g(x)=-x+1\le 0$ and $\Omega=\{x\in \mathbb{R}: x \ge 1\}.$
		Put $\mathcal{E}_1=[0,\frac{1}{2}], \mathcal{E}_2=[0,\frac{1}{2}]$. So $\theta=1$ and 
		$ \Omega_\varepsilon=\{x\in\mathbb{R}:x\ge 0\}.$ Let $z=0$. By  a simple calculation, one has
		\begin{align*}
			&\partial f_1^L(z)=\{1\}, \partial f_1^U(z)=\{1\},\\
			&\partial f_2^L(z)=\{1\}, \partial f_2^U(z)=\{1\},\\
			&\partial_x g(z,v)=\{-1\}, \ N(z;S)=(-\infty, 0].
		\end{align*}
        It is not difficult to show that $(z,\lambda) \in \Omega_{\mathcal{E}}\times \mathbb{R}_+$, where $\lambda=4$, is a KKT pair up to $\mathcal{E}$ of~\eqref{problem-R}. We now check the $\mathcal{E}$-pseudo-quasi generalized convexity on $S$ at $z$ of $(f,g)$. By choosing $w=(0,0)\in [N(z;S)]^\circ $, then it is easy to see that~\eqref{equ:3} are satisfied. 
        Thus, all the assumptions of Theorem~\ref{Sufficient-Theorem-II} are fulfilled.
        
		       Consider problem~\eqref{problem-S}, where $\varphi(x)=4x+1$ and $\Omega=\{x\in\mathbb{R}: x\ge 1\}$. It is clear that $z=0$ is an almost $\theta$-quasi solution of~\eqref{problem-S} as
		$$\varphi(z)\le \varphi(x)+\sqrt{\theta}\|x-z\|\ \forall x\in\Omega.$$
	 Therefore, according to Lemma~\ref{lemma-3.2} we conclude that $z$ is an almost $\mathcal{E}$-quasi Pareto solution~\eqref{problem-R}.
 \end{example}
\section{Approximate duality theorems of the Wolfe-type}\label{Duality-Relations}
For $y\in\mathbb{R}^n$ and $\lambda\in \mathbb{R}^p_+$, put
\begin{align*}
L(y, \lambda):=(L_1(y, \lambda), \ldots, L_m(y,\lambda))=\left(\left[L^L_1(y, \lambda), L^U_1(y, \lambda)\right], \ldots, \left[L^L_m(y, \lambda), L^U_m(y, \lambda)\right]\right),
\end{align*}
where 
\begin{align*}
L^L_i(y, \lambda):=f^L_i(y)+\frac{1}{2m}\sum_{j\in J}\lambda_jg_j(y), \ L^U_i(y, \lambda):=f^U_i(y)+\frac{1}{2m}\sum_{j\in J}\lambda_jg_j(y), \ \ j\in J.
\end{align*} 

By Theorem \ref{Existence-Theorem}, for a given $\mathcal{E}$, there exists a KKT pair up to $\mathcal{E}$, denoted by $(z_\mathcal{E},\lambda_\mathcal{E})$,  of~\eqref{problem-R}.     We now consider the following dual problem of the Wolfe-type for \eqref{problem-R} with respect to $(z_\mathcal{E},\lambda_\mathcal{E})$:
\begin{align}
\label{Dual-problem}\tag{RMD} 
 LU-&\mathrm{Max}\ \  L(y, \lambda) 
\\
&\text{subject to}\ \ (y, \lambda)\in\Omega_{D},\notag
\end{align}
where the feasible set is defined by
\begin{align*}
\Omega_{D}:=\Big\{(y, \lambda)\in &S\times\mathbb{R}^p_+\,:\, 0\in \sum_{i\in I}[\partial f_i^L(y)+\partial f_i^U(y)]+\sum_{j\in J}\lambda_j\mathrm{co}\,\{\partial_xg_j(y, v_j)\,:\, v_j\in \mathcal{V}_j(y)\}+
\\
&\ \ \ \ \  \  \ \ \ N(y; S)+\sqrt{\theta}\mathbb{B},
\lambda_j\leq \lambda_{j\mathcal{E}}, \ \ j\in J(\mathcal{E}):=\big\{j\in J: 0<g_j(z_\mathcal{E})\leq \sqrt{\theta}\big\}\Big\}
\end{align*} 
with $\lambda_\mathcal{E}:=(\lambda_{1\mathcal{E}}, \ldots, \lambda_{p\mathcal{E}})$.

It should be noticed that approximate  quasi Pareto solutions of the dual problem~\eqref{Dual-problem}  are defined similarly as in Definitions~\ref{Defi-solution-n1} and~\ref{Definition-2}  by replacing the relations $\leq_{LU}$ and $<_{LU}$ by  $\geq_{LU}$ and $>_{LU}$,  respectively. For example, we say that $(\bar y, \bar  \lambda)\in\Omega_{D}$ is an {\em  $\mathcal{E}$-quasi Pareto solution} of~\eqref{Dual-problem} if there is no $(y, \lambda)\in \Omega_{D}$ such that    
\begin{equation*}
L_i(y, \lambda)-\frac{\mathcal{E}_i}{\sqrt{\theta}}\|y-\bar y\|\geq_{LU} L_i(\bar y, \bar \lambda),\ \ \forall i\in I,
\end{equation*}
where at least one of the inequalities is strict.

The following theorem describes  duality relations for approximate quasi Pareto solutions between the primal problem \eqref{problem-R} and the dual one \eqref{Dual-problem}.
\begin{theorem}[$\mathcal{E}$-duality]\label{E-dual} Let $(z_\mathcal{E},\lambda_\mathcal{E})$ be a KKT pair up to $\mathcal{E}$ of \eqref{problem-R} and consider the dual problem \eqref{Dual-problem} with respect to $(z_\mathcal{E},\lambda_\mathcal{E})$. If $(f, g)$ is generalized convex on $S$ at every $y\in S$, then $(z_\mathcal{E},\lambda_\mathcal{E})$ is an $\mathcal{E}$-quasi  Pareto solution of \eqref{Dual-problem}.
\end{theorem}
\begin{proof} On the contrary, suppose that  $(z_\mathcal{E},\lambda_\mathcal{E})$ is not an $\mathcal{E}$-quasi  Pareto solution of~\eqref{Dual-problem}, then there exists $(y, \lambda)\in\Omega_{D}$ such that
\begin{equation*}
	L_i(y, \lambda)-\frac{\mathcal{E}_i}{\sqrt{\theta}}\|y-z_\mathcal{E}\|\geq_{LU} L_i(z_\mathcal{E},\lambda_\mathcal{E}),\ \ \forall i\in I,
\end{equation*}
where at least one of the inequalities is strict, or, equivalently,
\begin{align*}
f^L_i(y)+\frac{1}{2m}\sum_{j\in J}\lambda_jg_j(y)-\frac{\mathcal{E}^U_i}{\sqrt{\theta}}\|y-z_\mathcal{E}\|&\geq f^L_i(z_\mathcal{E})+\frac{1}{2m}\sum_{j\in J}\lambda_{j\mathcal{E}}g_j(z_\mathcal{E}),
\\
f^U_i(y)+\frac{1}{2m}\sum_{j\in J}\lambda_jg_j(y)-\frac{\mathcal{E}^L_i}{\sqrt{\theta}}\|y-z_\mathcal{E}\|&\geq f^U_i(z_\mathcal{E})+\frac{1}{2m}\sum_{j\in J}\lambda_{j\mathcal{E}}g_j(z_\mathcal{E}), \ \ \forall i\in I,
\end{align*}	
with at least one strict  inequality. Hence
\begin{equation}\label{equa:11-n}
\sum_{i\in I}[f^L_i(y)+f^U_i(y)]+\sum_{j\in J}\lambda_jg_j(y)-\sqrt{\theta}\|y-z_\mathcal{E}\|> \sum_{i\in I}[f^L_i(z_\mathcal{E})+f^U_i(z_\mathcal{E})]+\sum_{j\in J}\lambda_{j\mathcal{E}}g_j(z_\mathcal{E}).
\end{equation}
Since $(y, \lambda)\in \Omega_{D}$, there exist $z_i^{*L}\in\partial f_i^L(y)$,  $z_i^{*U}\in\partial f_i^U(y)$, $i\in I$,  $\lambda_{jk}\geq 0$, $x^*_{jk}\in\partial_x g_j(y, v_{jk})$, $v_{jk}\in\mathcal{V}_j(y)$, $k=1, \ldots, k_j$, $k_j\in\mathbb{N}$, $\sum_{k=1}^{k_j}\lambda_{jk}=1$,    $b^*\in \mathbb{B}$, and $\omega^* \in N(y; S)$ such that 
\begin{align} 
	&\sum_{i\in I}( z_i^{*L}+  z_i^{*U})+\sum_{j\in J} \sum_{k=1}^{k_j}\lambda_j\lambda_{jk}x^*_{jk} +\sqrt{\theta}b^*+\omega^*=0,\label{equa:8-n}
	\\
	& \lambda_j\leq \lambda_{j\mathcal{E}}, \ \ j\in J(\mathcal{E}).\label{equa:9-n}
\end{align}
By the generalized convexity of $(f, g)$ at $y$, there exists $\omega\in [N(y; S)]^\circ$ such that
\begin{equation}\label{equa:10-n} 
	\begin{split}
		&f_i^L(z_\mathcal{E}) -  f_i^L(y)\geq \langle z_i^{*L}, \omega\rangle, \ \ \forall i\in I,
		\\
		&f_i^U(z_\mathcal{E}) -  f_i^U(y)\geq \langle z_i^{*U}, \omega\rangle, \ \ \forall i\in I,
		\\
		&g_j(z_\mathcal{E}, v_{jk})- g_j(y, v_{jk})\geq  \langle x^*_{jk}, \omega\rangle,\ \   v_{jk}\in \mathcal{V}_j(y), j\in J,\ \ \text{and}\ \ \\
		& \langle b^*,\omega\rangle\leq \|z_\mathcal{E}-y\|,\ \ \forall b^*\in \mathbb{B}.
	\end{split}
\end{equation}
It follows from ~\eqref{equa:11-n},~\eqref{equa:8-n}, and~\eqref{equa:10-n}  that
\begin{align}
0&\leq \langle-\omega^*, \omega\rangle=\sum_{i\in I}[ \langle z_i^{*L}, \omega\rangle+  \langle z_i^{*U}, \omega\rangle]+\sum_{j\in J} \sum_{k=1}^{k_j}\lambda_j\lambda_{jk}\langle x^*_{jk}, \omega\rangle +\sqrt{\theta}\langle b^*, \omega\rangle\notag
\\
&\leq \sum_{i\in I}[f_i^L(z_\mathcal{E})+f_i^U(z_\mathcal{E})]-\sum_{i\in I}[f^L_i(y)+f^U_i(y)]+\sum_{j\in J}\sum_{k=1}^{k_j}\lambda_j\lambda_{jk}[g_j(z_\mathcal{E}, v_{jk})- g_j(y, v_{jk})]\notag
\\
&\ \ \ \ \ \ \ \ \ \ \ \ \ \ \ \ \ \ \ \ \ \ \ \ \ \ \ \ \ \ \ \ \ \ \ \ \ \ \ \ \ \ \ \ \ \ \ \ \ \ \ \ \ \ \ \ \ \ \ \ \ \ \ \ \ \ \ \ \ \ \ \ \ \ \ \ \ \ +\sqrt{\theta}\|z_\mathcal{E}-y\|\notag
\\
&<\sum_{j\in J}\lambda_jg_j(y)-\sum_{j\in J}\lambda_{j\mathcal{E}}g_j(z_\mathcal{E})+\sum_{j\in J}\sum_{k=1}^{k_j}\lambda_j\lambda_{jk}[g_j(z_\mathcal{E}, v_{jk})- g_j(y, v_{jk})]\notag
\\
&\leq\sum_{j\in J}\lambda_j\bigg[g_j(y)-\sum_{k=1}^{k_j}\lambda_{jk}g_j(y, v_{jk})\bigg]\!+\!\sum_{j\in J} \sum_{k=1}^{k_j}\lambda_{j\mathcal{E}}\lambda_{jk}g_j(z_\mathcal{E})\!-\!\sum_{j\in J}\lambda_{j\mathcal{E}}g_j(z_\mathcal{E}),\label{equa:13-n}
\end{align}
where the last inequality comes from \eqref{equa:9-n} and the fact that $g_j(z_\mathcal{E}, v_{j})\leq g_j(z_\mathcal{E})$ for all $v_{j}\in 
\mathcal{V}_j$ and $j\in J$. Since $v_{jk}\in \mathcal{V}_j(y)$ and $\sum_{k=1}^{k_j}\lambda_{jk}=1$, one has 
\begin{equation}\label{equa:14-n}
g_j(y)-\sum_{k=1}^{k_j}\lambda_{jk}g_j(y, v_{jk})=g_j(y)-\sum_{k=1}^{k_j}\lambda_{jk}g_j(y)=0,\ \ \forall j\in J.	
\end{equation} 
Clearly, 
\begin{equation}\label{equa:12-n}
\sum_{j\in J} \sum_{k=1}^{k_j}\lambda_{j\mathcal{E}}\lambda_{jk}g_j(z_\mathcal{E})-\sum_{j\in J}\lambda_{j\mathcal{E}}g_j(z_\mathcal{E})=\sum_{j\in J}\lambda_{j\mathcal{E}}g_j(z_\mathcal{E})\bigg( \sum_{k=1}^{k_j}\lambda_{jk}\bigg)-\sum_{j\in J}\lambda_{j\mathcal{E}}g_j(z_\mathcal{E})=0.
\end{equation}
Now, combining \eqref{equa:12-n}, \eqref{equa:14-n} and \eqref{equa:13-n}, we arrive at a contradiction. The proof is complete.
\end{proof}
The following example shows that the generalized convexity of $(f, g)$ in Theorem~\ref{E-dual} cannot be removed.
 \begin{example}\rm Let $f_1$, $f_2$, $S$, and $g$ be as in Example~\ref{ex3.3}.  It has been shown in Example~\ref{ex3.3} that $(f, g)$ is not generalized convex on $S$ at $z =0$. 
		Put $\mathcal{E}_1=[0,\frac{1}{2}], \mathcal{E}_2=[0,\frac{1}{2}]$. From our data, one has $g(x)=-x^2+1\le 0 $, and hence $\Omega=\{x\in S:x\ge 1\}.$
  
        We now consider the dual problem of~\eqref{problem-R} with respect to $(z_\mathcal{E},\lambda_\mathcal{E})$
	where the feasible set is defined by
	\begin{align*}
		\Omega_{D}:=\Big\{(y, \lambda)\in S\times\mathbb{R}^+\,:\, 0\in \sum_{i\in \{1,2\}}[\partial f_i^L(y)+\partial f_i^U(y)]+\lambda \mathrm{co}\,&\{\partial_xg(y, v)\,:\, v\in [0,1]\}
		\\
		&\ \ \ \ \ \ \ \ \  +N(y; S)+\sqrt{\theta}\mathbb{B}\Big\}
	\end{align*} 
    and the objective function
\begin{align*}
		L(y, \lambda):=(L_1(y, \lambda),L_2(y,\lambda))=\left(\left[L^L_1(y, \lambda), L^U_1(y, \lambda)\right], \left[L^L_2(y, \lambda), L^U_2(y, \lambda)\right]\right),
	\end{align*}
	where 
	\begin{align*}
		&L^L_1(y, \lambda):=y+\frac{1}{4}\lambda(-y^2+1), L^U_1(y, \lambda):=y+2+\frac{1}{4}\lambda(-y^2+1)\\
        &L^L_2(y, \lambda):=y+\frac{1}{4}\lambda(-y^2+1), L^U_2(y, \lambda):=y+1+\frac{1}{4}\lambda(-y^2+1).
	\end{align*} 
     Choose $ z_\mathcal{E}=\frac{1}{4},\lambda_\mathcal{E}=8$. It is obvious that
		\begin{align*}
			&\partial f_1^L(z_\mathcal{E})=\{1\}, \ \partial f_1^U(z_\mathcal{E})=\{1\},\\
			&\partial f_2^L(z_\mathcal{E})=\{1\}, \ \partial f_2^U(z_\mathcal{E})=\{1\},\\
			&\partial g(z_\mathcal{E})=\left\{-\frac{1}{2}\right\}.
		\end{align*}
		Hence, it is not hard to check that $ z_\mathcal{E}=\frac{1}{4},\lambda_\mathcal{E}=8$  is a KKT pair up to $\mathcal{E}$ of~\eqref{problem-R}.
        Moreover, in this case, $ z_\mathcal{E}=\frac{1}{4},\lambda_\mathcal{E}=8$ is not an $\mathcal{E}$-quasi Pareto solution of~\eqref{Dual-problem}. Indeed, suppose on the contrary that  $(z_\mathcal{E},\lambda_\mathcal{E})$ is an $\mathcal{E}$-quasi  Pareto solution of \eqref{Dual-problem}, then there is no $(y, \lambda)\in\Omega_{D}$ such that
		\begin{equation*}
			L_i(y, \lambda)-\frac{\mathcal{E}_i}{\sqrt{\theta}}\|y-z_\mathcal{E}\|\geq_{LU} L_i(z_\mathcal{E},\lambda_\mathcal{E}),\ \ \forall i\in \{1,2\},
		\end{equation*}
		where at least one of the inequalities is strict, or, equivalently,
		\begin{equation}\label{ct_n}
        \begin{split}
			f^L_i(y)+\frac{1}{4}\lambda g(y)-\frac{\mathcal{E}^U_i}{\sqrt{\theta}}\|y-z_\mathcal{E}\|&\geq f^L_i(z_\mathcal{E})+\frac{1}{4}\lambda_{\mathcal{E}} g (z_\mathcal{E}),
			\\
			f^U_i(y)+\frac{1}{4}\lambda g(y)-\frac{\mathcal{E}^L_i}{\sqrt{\theta}}\|y-z_\mathcal{E}\|&\geq f^U_i(z_\mathcal{E})+\frac{1}{4}\lambda_{\mathcal{E}}g_j(z_\mathcal{E}), \ \ \forall i\in \{1,2\},
		\end{split}	
        		\end{equation}
		with at least one strict  inequality. By the direct computation,  we have $L_1\left (\frac{1}{4},8\right)= \left[\frac{17}{8},\frac{33}{8}\right] $, $L_2\left(\frac{1}{4},8\right)= \left[\frac{17}{8},\frac{25}{8}\right]$, and we can find $ (y,\lambda)=(\frac{1}{8},16)\in \Omega_{D} $ such that~\eqref{ct_n} is satisfied, which is a contradiction. Therefore $(z_\mathcal{E},\lambda_\mathcal{E})=(\frac{1}{4},8) \in\Omega_{D}$ is not an $\mathcal{E}$-quasi  Pareto solution of~\eqref{Dual-problem} 
 	
 \end{example}
In the following, we establish the converse-like duality relation between approximate quasi Pareto solutions of the primal problem~\eqref{problem-R} and its corresponding dual problem~\eqref{Dual-problem}.
 \begin{theorem}[$\mathcal{E}$-converse like duality]\label{CE-dual} Let $(z_\mathcal{E},\lambda_\mathcal{E})$ be a KKT pair up to $\mathcal{E}$ of~\eqref{problem-R} and consider the dual problem~\eqref{Dual-problem} with respect to $(z_\mathcal{E},\lambda_\mathcal{E})$. Assume that $(y, \lambda)$ is a feasible solution of~\eqref{Dual-problem}.  If $(f, g)$ is generalized convex on $S$ at $y$ and the following condition~holds
\begin{equation}\label{equa:15-n}
g_j(x)\leq g_j(y), \ \ \forall x\in\Omega, j\in J,  
\end{equation} 	
then there is no $x\in\Omega$ such that
\begin{equation}\label{equa:16-n}
f_i(x)\leq_{LU} f_i(y) -\frac{\mathcal{E}_i}{\sqrt{\theta}}\|x-y\|,\ \ \forall i\in I,
\end{equation}
with at least one strict inequality.
 \end{theorem}
\begin{proof}
Suppose on the contrary that there exists $x\in\Omega$ satisfying \eqref{equa:16-n} with at least one strict inequality, i.e., 
\begin{equation*}
f^L_i(x)\leq f^L_i(y) -\frac{\mathcal{E}^U_i}{\sqrt{\theta}}\|x-y\| \ \ \text{and} \ \ f^U_i(x)\leq f^U_i(y) -\frac{\mathcal{E}^L_i}{\sqrt{\theta}}\|x-y\|, \ \  \forall i\in I,
\end{equation*}  
where at least one of the inequalities is strict. Hence
\begin{equation}\label{equa:17-n}
\sum_{i\in I}[f_i^L(x)-f_i^L(y)]+	\sum_{i\in I}[f_i^U(x)-f_i^U(y)]+\sqrt{\theta}\|x-y\|<0.
\end{equation}
Since $(y, \lambda)\in \Omega_{D}$, there exist $z_i^{*L}\in\partial f_i^L(y)$,  $z_i^{*U}\in\partial f_i^U(y)$, $i\in I$,  $\lambda_{jk}\geq 0$, $x^*_{jk}\in\partial_x g_j(y, v_{jk})$, $v_{jk}\in\mathcal{V}_j(y)$, $k=1, \ldots, k_j$, $k_j\in\mathbb{N}$, $\sum_{k=1}^{k_j}\lambda_{jk}=1$,    $b^*\in \mathbb{B}$, and $\omega^* \in N(y; S)$ such that 
\begin{equation} \label{equa:8-n-1}
\sum_{i\in I}( z_i^{*L}+  z_i^{*U})+\sum_{j\in J} \sum_{k=1}^{k_j}\lambda_j\lambda_{jk}x^*_{jk} +\sqrt{\theta}b^*+\omega^*=0.
\end{equation}
From the generalized convexity of $(f, g)$ at $y$, one can find $\omega\in [N(y; S)]^\circ$ such that
\begin{equation}\label{equa:10-n-1} 
	\begin{split}
		&f_i^L(x) -  f_i^L(y)\geq \langle z_i^{*L}, \omega\rangle, \ \ \forall i\in I,
		\\
		&f_i^U(x) -  f_i^U(y)\geq \langle z_i^{*U}, \omega\rangle, \ \ \forall i\in I,
		\\
		&g_j(x, v_{jk})- g_j(y, v_{jk})\geq  \langle x^*_{jk}, \omega\rangle,\ \   v_{jk}\in \mathcal{V}_j(y), j\in J,\ \ \text{and}\ \ \\
		& \langle b^*,\omega\rangle\leq \|x-y\|,\ \ \forall b^*\in \mathbb{B}.
	\end{split}
\end{equation}
Combining \eqref{equa:10-n-1}, \eqref{equa:8-n-1}, \eqref{equa:17-n}, and \eqref{equa:15-n}, we get a contradiction
\begin{align*}
	0&\leq \langle-\omega^*, \omega\rangle=\sum_{i\in I}[ \langle z_i^{*L}, \omega\rangle+  \langle z_i^{*U}, \omega\rangle]+\sum_{j\in J} \sum_{k=1}^{k_j}\lambda_j\lambda_{jk}\langle x^*_{jk}, \omega\rangle +\sqrt{\theta}\langle b^*, \omega\rangle\notag
	\\
	&\leq \sum_{i\in I}[f_i^L(x)-f_i^L(y)]+\sum_{i\in I}[f^U_i(x)-f^U_i(y)]+\sqrt{\theta}\|z_\mathcal{E}-y\|\notag
	\\
	&\ \ \ \ \ \ \ \ \ \ \ \ \ \ \ \ \ \ \ \ \ \  \ \ \ \ \ \ \ \ \ \ \ \ +\sum_{j\in J}\sum_{k=1}^{k_j}\lambda_j\lambda_{jk}[g_j(x, v_{jk})- g_j(y, v_{jk})]\notag
	\\
	&<\sum_{j\in J}\sum_{k=1}^{k_j}\lambda_j\lambda_{jk}[g_j(x, v_{jk})- g_j(y, v_{jk})]\leq \sum_{j\in J}\sum_{k=1}^{k_j}\lambda_j\lambda_{jk}[g_j(x)- g_j(y, v_{jk})] \notag
	\\
	&\leq \sum_{j\in J}\sum_{k=1}^{k_j}\lambda_j\lambda_{jk}[g_j(y)- g_j(y, v_{jk})]=0,
\end{align*}
where the last equality comes from the fact that $g_j(y)= g_j(y, v_{jk})$ for all $v_{jk}\in \mathcal{V}_j(y)$,  $k=1, \ldots, k_j$, and $j\in J$. The proof is complete.
\end{proof}
\section{Quasi-$\mathcal{E}$ Pareto saddle point} \label{section5}
Let $\mathcal{A}:=(A_1, \ldots, A_m)$ and $\mathcal{B}:=(B_1, \ldots, B_m)$, where $A_i$, $B_i$, $i\in I$, are intervals in $\mathcal{K}_c$. In what follows,   we use the following notations for convenience.
\begin{align*}
	\mathcal{A}&>_{LU}\mathcal{B} \Leftrightarrow
	\begin{cases}
		A_i\geq_{LU} B_i, \ \ \forall i\in I,
		\\
		A_k>_{LU} B_k,  \ \ \text{for at least one}\ \ k\in I. 
	\end{cases}
	\\
	\mathcal{A}&\ngtr_{LU}\mathcal{B} \ \ \text{is the negation of}\ \ \mathcal{A}>_{LU}\mathcal{B}.
\end{align*}

We consider an $\mathcal{E}$-interval-valued vector Lagrangian by setting
\begin{equation*}
\mathcal{L}_{\mathcal{E}}(x, \lambda, y, \mu):=\left([F_1^L(x, \lambda, y, \mu), F_1^U(x, \lambda, y, \mu)], \ldots, [F_m^L(x, \lambda, y, \mu), F_m^U(x, \lambda, y, \mu)]\right),
\end{equation*}
\begin{align*}
&F_i^L(x, \lambda, y, \mu):=f^L_i(x)+\frac{1}{2m}\sum_{j\in J}\lambda_jg_j(x)+\frac{\mathcal{E}^L_i}{\sqrt{\theta}}\|x-y\|-\frac{\mathcal{E}^L_i}{\sqrt{\theta}}\|\lambda-\mu\|_1
\\
&F_i^U(x, \lambda, y, \mu):=f^U_i(x)+\frac{1}{2m}\sum_{j\in J}\lambda_jg_j(x)+\frac{\mathcal{E}^U_i}{\sqrt{\theta}}\|x-y\|-\frac{\mathcal{E}^U_i}{\sqrt{\theta}}\|\lambda-\mu\|_1, \ \ i\in I,
\end{align*}
where $(y,\mu)\in \mathbb{R}^n\times\mathbb{R}^p$ and 
$$\|\lambda-\mu\|_1:=\sum_{j\in J}|\lambda_j-\mu_j|.$$
\begin{definition}
\rm 
A point $(\bar x, \bar\lambda)\in\mathbb{R}\times\mathbb{R}^p_+$ is called a {\em quasi-$\mathcal{E}$ Pareto saddle point} of $\mathcal{L}_{\mathcal{E}}$ if the following conditions hold:
\begin{enumerate}[(i)]
	\item $\mathcal{L}_{\mathcal{E}}(\bar x, \lambda, \bar x, \bar \lambda)\ngtr_{LU} \mathcal{L}_{\mathcal{E}}(\bar x, \bar\lambda, \bar x, \bar \lambda)$ for all $\lambda\in \mathbb{R}^p_+$;
	\item   $\mathcal{L}_{\mathcal{E}}(\bar x, \bar\lambda, \bar x, \bar \lambda)\ngtr_{LU} \mathcal{L}_{\mathcal{E}}(x, \bar\lambda, \bar x, \bar \lambda)$ for all $x\in \mathbb{R}^n$.
\end{enumerate}
\end{definition}
\begin{theorem}\label{theo5.1} Let $(z_\mathcal{E},\lambda_\mathcal{E})$ be a KKT pair up to $\mathcal{E}$ of \eqref{problem-R}. If $(f, g)$ is generalized convex on $S$ at $z_\mathcal{E}$, then $(z_\mathcal{E},\lambda_\mathcal{E})$  is a quasi-$\mathcal{E}$ Pareto saddle point of $\mathcal{L}_{\mathcal{E}}$.
\end{theorem}
\begin{proof} Suppose  that there exists $\lambda\in \mathbb{R}^p_+$    such that
\begin{equation}
\mathcal{L}_{\mathcal{E}}(z_\mathcal{E}, \lambda, z_\mathcal{E}, \lambda_\mathcal{E})>_{LU} \mathcal{L}_{\mathcal{E}}(z_\mathcal{E}, \lambda_\mathcal{E}, z_\mathcal{E}, \lambda_\mathcal{E}).\label{equa-L-1}
\end{equation}
This means that
\begin{equation*}
F_i^L(z_\mathcal{E}, \lambda, z_\mathcal{E}, \lambda_\mathcal{E})\geq F_i^L(z_\mathcal{E}, \lambda_\mathcal{E}, z_\mathcal{E}, \lambda_\mathcal{E}) \ \ \text{and}\ \ F_i^U(z_\mathcal{E}, \lambda, z_\mathcal{E}, \lambda_\mathcal{E})\geq F_i^U(z_\mathcal{E}, \lambda_\mathcal{E}, z_\mathcal{E}, \lambda_\mathcal{E}), \ \ \forall i\in I, 
\end{equation*}
where at least one of the inequalities is strict. This means that
\begin{align*}
&\frac{1}{2m}\sum_{j\in J}\lambda_jg_j(z_\mathcal{E}) -\frac{\mathcal{E}^L_i}{\sqrt{\theta}}\|\lambda-\lambda_\mathcal{E}\|_1\geq \frac{1}{2m}\sum_{j\in J}\lambda_{j\mathcal{E}}g_j(z_\mathcal{E}),
\\
&\frac{1}{2m}\sum_{j\in J}\lambda_jg_j(z_\mathcal{E}) -\frac{\mathcal{E}^U_i}{\sqrt{\theta}}\|\lambda-\lambda_\mathcal{E}\|_1\geq \frac{1}{2m}\sum_{j\in J}\lambda_{j\mathcal{E}}g_j(z_\mathcal{E}), \ \ \forall i\in I,
\end{align*}
with at least one strict  inequality. Hence
\begin{equation*}
\sum_{j\in J}\lambda_jg_j(z_\mathcal{E})-\sqrt{\theta}\|\lambda-\lambda_\mathcal{E}\|_1>\sum_{j\in J}\lambda_{j\mathcal{E}}g_j(z_\mathcal{E}),
\end{equation*}
or, equivalently,
\begin{equation}\label{equa-L-3}
\sum_{j\in J}(\lambda_j-\lambda_{j\mathcal{E}})g_j(z_\mathcal{E})>\sqrt{\theta}\|\lambda-\lambda_\mathcal{E}\|_1.
\end{equation}
By definition of $J(\mathcal{E})$, we see that $\lambda_{j\mathcal{E}}=0$ and $\lambda_j g_j(z_\mathcal{E})\leq 0$ for all $j\notin J(\mathcal{E})$. Hence
\begin{align*}
\sum_{j\in J}(\lambda_j-\lambda_{j\mathcal{E}})g_j(z_\mathcal{E})&\leq \sum_{j\in J(\mathcal{E})}(\lambda_j-\lambda_{j\mathcal{E}})g_j(z_\mathcal{E})
\\
&\leq \sum_{j\in J(\mathcal{E})}|\lambda_j-\lambda_{j\mathcal{E}}|\sqrt{\theta}\leq \sqrt{\theta}\|\lambda-\lambda_\mathcal{E}\|_1,
\end{align*}
contrary to \eqref{equa-L-3}. This means that there is no $\lambda\in \mathbb{R}^p_+$ satisfying \eqref{equa-L-1}.

We now assume that there exists $x\in\mathbb{R}^n$ such that
\begin{equation*}
\mathcal{L}_{\mathcal{E}}(z_\mathcal{E}, \lambda_\mathcal{E}, z_\mathcal{E}, \lambda_\mathcal{E})>_{LU} \mathcal{L}_{\mathcal{E}}(x, \lambda_\mathcal{E}, z_\mathcal{E}, \lambda_\mathcal{E}), 
\end{equation*} 
or, equivalently,
\begin{align*}
f^L_i(z_\mathcal{E})+\frac{1}{2m}\sum_{j\in J}\lambda_{j\mathcal{E}}g_j(z_\mathcal{E})&\geq f^L_i(x)+\frac{1}{2m}\sum_{j\in J}\lambda_{j\mathcal{E}}g_j(x)+\frac{\mathcal{E}^L_i}{\sqrt{\theta}}\|z_\mathcal{E}-x\|
\\
f^U_i(z_\mathcal{E})+\frac{1}{2m}\sum_{j\in J}\lambda_{j\mathcal{E}}g_j(z_\mathcal{E})&\geq f^U_i(x)+\frac{1}{2m}\sum_{j\in J}\lambda_{j\mathcal{E}}g_j(x)+\frac{\mathcal{E}^U_i}{\sqrt{\theta}}\|z_\mathcal{E}-x\|, \ \ \forall i\in I,
\end{align*}
where at least one of the inequalities is strict. This implies that
\begin{equation*}
\sum_{i\in I}[f^L_i(x)-f^L_i(z_\mathcal{E})]+\sum_{i\in I}[f^U_i(x)-f^U_i(z_\mathcal{E})]+\sum_{j\in J}\lambda_{j\mathcal{E}}[g_j(x)-g_j(z_\mathcal{E})]+\sqrt{\theta}\|z_\mathcal{E}-x\|<0.
\end{equation*}
Since  $(z_\mathcal{E},\lambda_\mathcal{E})$ is a KKT pair up to $\mathcal{E}$ of \eqref{problem-R}, there exist $z_i^{*L}\in\partial f_i^L(z_\mathcal{E})$,  $z_i^{*U}\in\partial f_i^U(z_\mathcal{E})$, $i\in I$,  $\lambda_{jk}\geq 0$, $x^*_{jk}\in\partial_x g_j(z_\mathcal{E}, v_{jk})$, $v_{jk}\in\mathcal{V}_j(z_\mathcal{E})$, $k=1, \ldots, k_j$, $k_j\in\mathbb{N}$, $\sum_{k=1}^{k_j}\lambda_{jk}=1$,    $b^*\in \mathbb{B}_{\mathbb{R}^n}$, and $\omega^* \in N(z_\mathcal{E}; S)$ such that 
\begin{equation} \label{equa-L-4}
	\sum_{i\in I}( z_i^{*L}+  z_i^{*U})+\sum_{j\in J} \sum_{k=1}^{k_j}\lambda_{j\mathcal{E}}\lambda_{jk}x^*_{jk} +\sqrt{\theta}b^*+\omega^*=0.
\end{equation}
By the generalized convexity of $(f, g)$ at $z_\mathcal{E}$, there exists $\omega\in [N(z_\mathcal{E}; S)]^\circ$ such that
\begin{equation}\label{equa-L-5} 
	\begin{split}
		&f_i^L(x) -  f_i^L(z_\mathcal{E})\geq \langle z_i^{*L}, \omega\rangle, \ \ \forall i\in I,
		\\
		&f_i^U(x) -  f_i^U(z_\mathcal{E})\geq \langle z_i^{*U}, \omega\rangle, \ \ \forall i\in I,
		\\
		&g_j(x, v_{jk})- g_j(z_\mathcal{E}, v_{jk})\geq  \langle x^*_{jk}, \omega\rangle,\ \   v_{jk}\in \mathcal{V}_j(z_\mathcal{E}), j\in J,\ \ \text{and}\ \ \\
		& \langle b^*,\omega\rangle\leq \|x-z_\mathcal{E}\|,\ \ \forall b^*\in \mathbb{B}.
	\end{split}
\end{equation}
Combining \eqref{equa-L-4} and \eqref{equa-L-5}, we get  
\begin{align*}
	0&\leq \langle-\omega^*, \omega\rangle=\sum_{i\in I}[ \langle z_i^{*L}, \omega\rangle+  \langle z_i^{*U}, \omega\rangle]+\sum_{j\in J} \sum_{k=1}^{k_j}\lambda_{j\mathcal{E}}\lambda_{jk}\langle x^*_{jk}, \omega\rangle +\sqrt{\theta}\langle b^*, \omega\rangle\notag
	\\
	&\leq \sum_{i\in I}[f_i^L(x)-f_i^L(z_\mathcal{E})]+\sum_{i\in I}[f^U_i(x)-f^U_i(z_\mathcal{E})]+\sqrt{\theta}\|x-z_\mathcal{E}\|\notag
	\\
	&\ \ \ \ \ \ \ \ \ \ \ \ \ \ \ \ \ \ \ \ \ \  \ \ \ \ \ \ \ \ \ \ \ \ +\sum_{j\in J}\sum_{k=1}^{k_j}\lambda_{j\mathcal{E}}\lambda_{jk}[g_j(x, v_{jk})- g_j(z_\mathcal{E}, v_{jk})]\notag
	\\
	&<\sum_{j\in J}\sum_{k=1}^{k_j}\lambda_{j\mathcal{E}}\lambda_{jk}[g_j(x, v_{jk})- g_j(z_\mathcal{E}, v_{jk})]-\sum_{j\in J}\lambda_{j\mathcal{E}}[g_j(x)-g_j(z_\mathcal{E})] \notag
	\\
	&= \sum_{j\in J(\mathcal{E})}\sum_{k=1}^{k_j}\lambda_{j\mathcal{E}}\lambda_{jk}[g_j(x, v_{jk})- g_j(z_\mathcal{E}, v_{jk})]-\sum_{j\in J(\mathcal{E})}\lambda_{j\mathcal{E}}[g_j(x)-g_j(z_\mathcal{E})]
	\\
	&\leq \sum_{j\in J(\mathcal{E})}\sum_{k=1}^{k_j}\lambda_{j\mathcal{E}}\lambda_{jk}[g_j(x)- g_j(z_\mathcal{E})]-\sum_{j\in J(\mathcal{E})}\lambda_{j\mathcal{E}}[g_j(x)-g_j(z_\mathcal{E})]=0,
\end{align*}
a contradiction. The proof is complete.
\end{proof}
The following example shows that the generalized convexity of $(f,g)$
on $S$ used in Theorem~\ref{theo5.1} cannot be omitted.
\begin{example}\rm  Let $f=(f_1,f_2)$  be defined by
		$f_1(x)=[2x,2x+1]$,
		$f_2(x)=[-x^2,-x^2+1], $ $S=\{x\in\mathbb{R}:x\ge -3\}$, and $g(x,v)=-x-v+1, v\in[0,1]$.
		Put $\mathcal{E}_1=[0,\frac{1}{2}], \mathcal{E}_2=[0,\frac{1}{2}]$. 
       It is not difficult to show that $(f, g)$ is not generalized convex on $S$ at $z =0$.        
      From the data, we get $g(x)=-x+1\le 0$ and $  \Omega=\{x\in S:x\ge 1\}.$

      Consider at the point $\left( {x_{\mathcal{E}}} ,{\lambda_{\mathcal{E}}}\right)=\left( 0,4 \right)$. By the direct computation, we have:
    \begin{align*}
			&\partial f_1^L(x_{\mathcal{E}})=\{2\}, \ \partial f_1^U(x_{\mathcal{E}})=\{2\},\\
			&\partial f_2^L(x_{\mathcal{E}})=\{0\},\  \partial f_2^U(x_{\mathcal{E}})=\{0\},\\
			&\partial g(z)=\{-1\}, N(x_{\mathcal{E}}; S)=\{0\},
		\end{align*} 
and it is easy to see that $\left( {x_{\mathcal{E}}} ,{\lambda_{\mathcal{E}}}\right)=\left( 0,4 \right)$   is a KKT pair up to $\mathcal{E}$ of~\eqref{problem-R}.

      We consider an $\mathcal{E}$-interval-valued vector Lagrangian by setting
	\begin{equation*}
		\mathcal{L}_{\mathcal{E}}(x, \lambda, y, \mu):=\left([F_1^L(x, \lambda, y, \mu), F_1^U(x, \lambda, y, \mu)], [F_2^L(x, \lambda, y, \mu), F_2^U(x, \lambda, y, \mu)]\right),
	\end{equation*}
	\begin{align*}
		&F_1^L(x, \lambda, y, \mu):=2x+\frac{1}{4}\lambda(-x+1)+\frac{0}{\sqrt{\theta}}\|x-y\|-\frac{0}{\sqrt{\theta}}\|\lambda-\mu\|_1,
		\\
		&F_2^L(x, \lambda, y, \mu):=-x^2+\frac{1}{4}\lambda(-x+1)+\frac{0}{\sqrt{\theta}}\|x-y\|-\frac{0}{\sqrt{\theta}}\|\lambda-\mu\|_1,
		\\&F_1^U(x, \lambda, y, \mu):=2x+1+\frac{1}{4}\lambda(-x+1)+\frac{1}{2}\|x-y\|-\frac{1}{2}\|\lambda-\mu\|_1, \\
	&F_2^U(x, \lambda, y, \mu):=-x^2+1+\frac{1}{4}\lambda(-x+1)+\frac{1}{2}\|x-y\|-\frac{1}{2}\|\lambda-\mu\|_1, 
    \end{align*}
	where $(y,\mu)\in \mathbb{R}\times\mathbb{R}$ and 
	$\|\lambda-\mu\|_1:=|\lambda_j-\mu_j|.$ At the point  $\left( {x_{\mathcal{E}}} ,{\lambda_{\mathcal{E}}}\right)=\left( 0,4 \right)$, one has $$\mathcal{L}_{\mathcal{E}}(0,4,0,4)=([1,2],[1,2]).$$ We now consider  the point $\left( x,\bar{\lambda} \right)=\left( -2,4 \right)$, then
     $$\mathcal{L}_{\mathcal{E}}(-2,4,0,4)=([-1,1],[-1,1]).$$
  So, there exists $x\in\mathbb{R}^n$ such that
		\begin{equation*}
			\mathcal{L}_{\mathcal{E}}(\bar{x},\bar{\lambda}, \bar{x},\bar{\lambda})>_{LU} \mathcal{L}_{\mathcal{E}}(x, \bar{\lambda}, \bar{x},\bar{\lambda}).
		\end{equation*}  
 In the other words,   $(\bar{x},\bar{\lambda})$  is not a quasi-$\mathcal{E}$ Pareto saddle point of $\mathcal{L}_{\mathcal{E}}$.
\end{example}
We close this section by presenting converse-like duality relations for the quasi-$\mathcal{E}$ Pareto saddle point of $\mathcal{L}_{\mathcal{E}}$ and almost $\mathcal{E}$-quasi Pareto solution of \eqref{problem-R}.
\begin{theorem} If $(z_\mathcal{E},\lambda_\mathcal{E})$ is a quasi-$\mathcal{E}$ Pareto saddle point of $\mathcal{L}_{\mathcal{E}}$ and $g_j(x)\leq g_j(z_\mathcal{E})$ for all $x\in S$ and $j\in J$, then $z_\mathcal{E}$ is an almost $\mathcal{E}$-quasi Pareto solution of \eqref{problem-R}.
\end{theorem}
\begin{proof}
Since $(z_\mathcal{E},\lambda_\mathcal{E})$ is a quasi-$\mathcal{E}$ Pareto saddle point of $\mathcal{L}_{\mathcal{E}}$, there is no $\lambda\in \mathbb{R}^p_+$ such~that
		\begin{equation}\label{eq5.1}
			\sum_{j\in J}\lambda_jg_j(z_\mathcal{E})-\sqrt{\theta}\|\lambda-\lambda_\mathcal{E}\|_1>\sum_{j\in J}\lambda_{j\mathcal{E}}g_j(z_\mathcal{E}),
		\end{equation}
		or, equivalently,
		\begin{equation*}
			\sum_{j\in J}(\lambda_j-\lambda_{j\mathcal{E}})g_j(z_\mathcal{E})>\sqrt{\theta}\|\lambda-\lambda_\mathcal{E}\|_1.
		\end{equation*}
		If $z_\mathcal{E}\notin \Omega_\mathcal{E}$, then $g_k(z_\mathcal{E})>0$ for some $k\in J$. So, we get
        \begin{align*}
&g_k(z_\mathcal{E})+\lambda_{k\mathcal{E}}g_k(z_\mathcal{E})>\theta+\lambda_{k\mathcal{E}}g_k(z_\mathcal{E})\\
&\sum_{k\neq j}\lambda_{j\mathcal{E}}g_j(z_\mathcal{E})+(1+\lambda_{k\mathcal{E}})g_jz_\mathcal{E}> \theta+ \sum_{j\in J}\lambda_{j\mathcal{E}}g_j(z_\mathcal{E}).
		\end{align*}
        Choosing   $\bar{\lambda_k}=1+\lambda_{k\mathcal{E}} $ and $\bar{\lambda_j}=\lambda_{j\mathcal{E}},$
 for all $j\not= k$, one obtains       \begin{equation*}
			\sum_{j\in J}\bar{\lambda_j}g_j(z_\mathcal{E})-\theta\|\bar{\lambda}-\lambda_\mathcal{E}\|_1>\sum_{j\in J}\lambda_{j\mathcal{E}}g_j(z_\mathcal{E}),
		\end{equation*}
        which contradicts~\eqref{eq5.1}. Therefore, we conclude that $z_\mathcal{E}\in \Omega_\mathcal{E}$.
        
		Now, we use the other inequality for a quasi-$\mathcal{E}$-Pareto-saddle point. Since $(z_\mathcal{E},\lambda_\mathcal{E})$ is a quasi-$\mathcal{E}$ Pareto saddle point of $\mathcal{L}_{\mathcal{E}}$ then there is not  $x\in\mathbb{R}^n$ such that
		\begin{equation*}
			\mathcal{L}_{\mathcal{E}}(z_\mathcal{E}, \lambda_\mathcal{E}, z_\mathcal{E}, \lambda_\mathcal{E})>_{LU} \mathcal{L}_{\mathcal{E}}(x, \lambda_\mathcal{E}, z_\mathcal{E}, \lambda_\mathcal{E}),\label{equa-L-2}
		\end{equation*} 
		or, equivalently,
		\begin{align*}
			f^L_i(z_\mathcal{E})+\frac{1}{2m}\sum_{j\in J}\lambda_{j\mathcal{E}}g_j(z_\mathcal{E})&\geq f^L_i(x)+\frac{1}{2m}\sum_{j\in J}\lambda_{j\mathcal{E}}g_j(x)+\frac{\mathcal{E}^L_i}{\sqrt{\theta}}\|z_\mathcal{E}-x\|
			\\
			f^U_i(z_\mathcal{E})+\frac{1}{2m}\sum_{j\in J}\lambda_{j\mathcal{E}}g_j(z_\mathcal{E})&\geq f^U_i(x)+\frac{1}{2m}\sum_{j\in J}\lambda_{j\mathcal{E}}g_j(x)+\frac{\mathcal{E}^U_i}{\sqrt{\theta}}\|z_\mathcal{E}-x\|, \ \ \forall i\in I,
		\end{align*}
		where at least one of the inequalities is strict.
		However, $g_j(x)\leq g_j(z_\mathcal{E})$, for $j\in J$ and $x\in \Omega$. 
        Thus, we conclude that there is no $x\in \Omega$ such that
		\begin{align*}
			f^L_i(z_\mathcal{E})&\geq f^L_i(x)+\frac{\mathcal{E}^L_i}{\sqrt{\theta}}\|z_\mathcal{E}-x\|
			\\
			f^U_i(z_\mathcal{E})&\geq f^U_i(x)+\frac{\mathcal{E}^U_i}{\sqrt{\theta}}\|z_\mathcal{E}-x\|, \ \ \forall i\in I,
		\end{align*}
        with at least one strict inequality. In the other words,
    $z_\mathcal{E}$ is an almost $\mathcal{E}$-quasi Pareto solution of \eqref{problem-R}. The proof is complete.
\end{proof}

\section{Conclusions}\label{Conclusions}
 In this paper, we focus on studying approximate Pareto solutions of nonsmooth interval-valued multiobjective optimization problems with data uncertainty in constraints.  On the one hand, we introduce some  types of approximate  Pareto solutions for the robust counterpart \eqref{problem-R} of the problem in question. By using a scalar penalty function, we obtain a result on the existence of an almost regular $\mathcal{E}$-Pareto solution of \eqref{problem-R} that satisfies the KKT necessary optimality condition up to a given precision. We then present sufficient optimality conditions under the assumptions of generalized convexity.  On the other hand, we suggest establishing Wolfe-type $\mathcal{E}$-duality theorems for  approximate solutions of the considered problem. We also consider the $\mathcal{E}$-interval-valued vector Lagrangian function and the quasi-$\mathcal{E}$ Pareto saddle point.

\section*{Disclosure statement} 
The authors declare that they have no conflict of interest.

\section*{Data availability} There is no data included in this paper.




\begin{thebibliography}{99}
	
	
	
	\bibitem{Alefeld}
	 Alefeld G, Herzberger J. Introduction to interval computations. New York (NY): Academic Press; 1983.  
	
	
	
	\bibitem{Ben-Nemirovski-98}
	Ben-Tal A, Nemirovski A. Robust convex optimization. Math. Oper. Res. 1998;23:769--805.
	
	\bibitem{Ben-Nemirovski-08}
	 Ben-Tal A, Nemirovski A. A selected topics in robust convex optimization. Math. Program. 2008;112: 125--158.
	
	\bibitem{Ben-Nemirovski-09}
	Ben-Tal A, Ghaoui LE, Nemirovski A.
	 Robust Optimization. Princeton (NJ): Princeton University Press; 2009.
	
	
	
	
	
	\bibitem{Chalco-Cano et.al.-13}
	 Chalco-Cano Y, Lodwick WA., Rufian-Lizana A. Optimality conditions of type KKT for optimization problem with interval-valued objective function via generalized derivative. {Fuzzy Optim. Decis. Mak.} 2013;12:305--322.  

     \bibitem{Chen-19} Chen J, K\"obis E, Yao JC. Optimality conditions and duality for robust nonsmooth multiobjective optimization problems with constraints. J Optim Theory Appl. 2019;181:411--436.
	
	\bibitem{Chuong-16}
    Chuong TD. Optimality and duality for robust multiobjective optimization problems. {Nonlinear Anal.} 2016;134:127--143.
	\bibitem{Chuong-14}   Chuong TD, Kim D. Nonsmooth semi-infinite multiobjective optimization problems. J. Optim.
	Theory Appl. 2014;160:748--762.
	
	
	\bibitem{Chuong-Kim-16}
    Chuong TD, Kim DS. Approximate solutions of multiobjective optimization problems. {Positivity.} 2016;20:187--207.
    
    
	
	\bibitem{Fakhar-19}
	 Fakhar M, Mahyarinia, Zafarani J. On approximate solutions for nonsmooth robust multiobjective optimization problems. Optimization. 2019;68:1653--1683.

\bibitem{Gutierrez} Guti\'errez C, Huerga L, K\"obis E, Tammer C. A scalarization scheme for binary relations with applications to set-valued and robust optimization. J Global Optim. 2021;79:233--256. 

	\bibitem{Hung-Tuan-Tuyen}  Hung NH, Tuan HN, Tuyen NV. On approximate quasi Pareto solutions in nonsmooth semi-infinite interval-valued vector optimization problems. Appl. Anal. 2023;102:2432--2448.
	
	
	\bibitem{Hung-Tuyen} Hung NH, Tuyen NV. Optimality conditions and duality relations in nonsmooth fractional interval-valued multiobjective optimization, Appl. Set-Valued Anal. Optim. 2023;5:31--47.
	
	\bibitem{Ioffe79}  Ioffe AD, Tihomirov VM. Theory of extremal problems. Amsterdam: North-Holland; 1979.
	
	\bibitem{Ishibuchi-Tanaka-90}
	Ishibuchi H, Tanaka H. Multiobjective programming in optimization of the interval objective function. {European J. Oper. Res.}  1990;48:219--225.   
	
	
	\bibitem{Jennane}
	 Jennane M, Kalmoun EM, Lafhim L. Optimality conditions for nonsmooth interval-valued and multiobjective semi-infinite programming. RAIRO-Oper. Res. 2021;55:1--11.
	\bibitem{Jiao-Liguo-21}
	 Jiao LG, Kim DS, Zhou Y. Quasi $\epsilon $-solutions in a semi-infinite programming problem with locally Lipschitz data.  Optim. Lett. 2021;15:1759--1772.


	\bibitem{Kim-Son-18}
	 Kim DS, Son, TQ. An approach to $\epsilon$-duality theorems for nonconvex semi-infinite multiobjective optimization problems. Taiwanese J. Math. 2018;22:1261--1287.

     \bibitem{Klamroth}
Klamroth K, K\"obis E, Sch\"obel A, Tammer C. A unified approach for different concepts of robustness and stochastic programming via nonlinear scalarizing functionals. Optimization. 2013;62:649--671.

\bibitem{Kobis-14} K\"obis E. On robust optimization: a unified approach to robustness using a nonlinear scalarizing functional and relations to set optimization [PhD thesis]. Martin-Luther-University Halle-Wittenberg; 2014. 

\bibitem{Kobis-15} K\"obis E. On robust optimization: relations between scalar robust optimization and unconstrained multicriteria optimization. J. Optim. Theory Appl. 2015;167:969--984.

\bibitem{Kobis-17}  K\"obis E, Tammer C.  Robust vector optimization with a variable domination structure. Carpathian J. Math. 2017;33:343--351.
	 
	 \bibitem{Kulisch-81} Kulisch UW, Miranker WL. Computer Arithmetic in Theory and Practice. New  York (NY): Academic Press; 1981.
	
	\bibitem{Kumar}
	  Kumar P, Sharma B, Daga J. Interval-valued programming problem with infinite constraints. J. Oper. Res. Soc. China. 2018;6:611--626.
	 
	 \bibitem{Liu-91}  Liu JC. $\epsilon$-duality theorem of nondifferentiable nonconvex multiobjective programming. J. Optim. Theory Appl. 1991;69:153--167.
	 
	\bibitem{Loridan-82}
	 Loridan P. Necessary conditions for $\epsilon$-optimality. {Math. Program. Study}. 1982;19:140--152 .
	\bibitem{Loridan-84}
	 Loridan P. $\epsilon$-solutions in vector minimization problems. {J. Optim. Theory Appl.} 1984;43:265--276.
	
	\bibitem{Moore-1966}
	 Moore RE. Interval analysis. Englewood Cliffs, NJ: Prentice-Hall; 1966.
	
	\bibitem{Moore-1979}
	 Moore RE. Methods and applications of interval analysis. Philadelphia: SIAM Studies in Applied Mathematics; 1979.
	
	\bibitem{mor06}
	 Mordukhovich BS. Variational analysis and generalized differentiation, I basic theory, II applications. Berlin: Springer-Verlag; 2006. 
	
	
	\bibitem{Qian}
	 Qian X,  Wang KR, Li XB. Solving vector interval-valued optimization problems with infinite interval constraints via integral-type penalty function. Optimization. 2022;71:3627--3645.
	
	\bibitem{Rahimi}
	 Rahimi M, Soleimani-Damaneh M. Robustness in deterministic vector optimization. J. Optim. Theory Appl. 2018;179:137--162.
	

	

\bibitem{Saadati-Oveisiha}  Saadati M, Oveisiha M. Approximate solutions for robust multiobjective optimization programming in Asplund spaces. Optimization. 2024; 73:329--357.
	
	\bibitem{Singh-Dar-15}
	 Singh AD, Dar BA. Optimality conditions in multiobjective programming problems with interval valued objective functions. {Control Cybern.} 2015;44:19--45.
	
	
	
	\bibitem{Singh-Dar-Kim-19}
	 Singh D, Dar BA, Kim DS. Sufficiency and duality in non-smooth interval valued programming problems. {J. Ind. Manag. Optim.} 2019;15:647--665.
	
\bibitem{Son-09} Son TQ, Strodiot JJ, Nguyen VH. $\varepsilon$-Optimality and $\varepsilon$-Lagrangian duality for a nonconvex programming problem with an infinite number of constraints. J. Optim. Theory Appl. 2009;141:389--409.

\bibitem{Son-13} Son TQ, Kim DS.  $\varepsilon$-Mixed duality for nonconvex multiobjective programs with an infinite number of constraints. J. Glob. Optim. 2013;57:447--465.
	
	
	\bibitem{Son-Tuyen-Wen-19}
	 Son TQ, Tuyen NV, Wen CF. Optimality conditions for approximate Pareto solutions of a nonsmooth vector optimization problem with an infinite number of constraints. {Acta Math. Vietnam.} 2020;45:435--448.  
	 
	\bibitem{Son-24}  Son TQ, Bao HK, Kim DS. Approximate optimality conditions for nonsmooth optimization problems. Taiwanese J. Math. 2024;28:1245--1266.
	
	\bibitem{Su-2020}
	 Su TV, Dinh DH. Duality results for interval-valued pseudoconvex optimization problem with equilibrium constraints with applications. Comput. Appl. Math. 2020;39:1--24.

    \bibitem{Sun-Teo-Zeng-Liu}  Sun X, Teo KL, Zeng J, Liu L. Robust approximate optimal solutions for nonlinear semi-infinite programming with uncertainty. Optimization. 2020;69:2109--2129.
	
	
	
	\bibitem{Tung-2019}
	 Tung, L.T. Karush--Kuhn--Tucker optimality conditions and duality for semi-infinite programming with multiple interval-valued objective functions. {J. Nonlinear Funct. Anal.}  {\bf 2019},  1--21 (2019)  
	   
	   \bibitem{TXS-20}
	   Tuyen NV, Xiao Y-B, Son TQ. On approximate KKT optimality conditions for cone-constrained vector optimization problems. {J. Nonlinear Convex Anal.} 2020;21:105--117.
	
	\bibitem{Tuyen-2021}
	 Tuyen NV. Approximate solutions of interval-valued optimization problems. Investigaci\'on Oper. 2021; 42:223--237.
	
	  
	 	\bibitem{Tuyen-2022}  Tuyen NV, Wen C-F, Son TQ. An approach to characterizing  $\epsilon$-solution sets of convex programs. TOP. 2022;30:249--269.
	
	\bibitem{Wu-07}
	 Wu HC. The Karush--Kuhn--Tuker optimality conditions in an optimization problem with interval valued objective functions. {Eur. J. Oper. Res.} 2007;176:46--59.
	
	\bibitem{Wu-09}
 Wu HC. The Karush--Kuhn--Tucker optimality conditions in multiobjective programming problems with interval-valued objective functions. {Eur. J. Oper. Res.} 2009;196:49--60.
	
	  
	\bibitem{Wu-09-c}
	 Wu HC. Duality theory for optimization problems with interval-valued objective functions. J. Optim.
	Theory Appl. 2010;144:615--628.
	
\end{thebibliography}
\end{document}